\documentclass[a4paper]{amsart}
\usepackage[utf8]{inputenc}

\usepackage{amssymb}
\usepackage{float}
\usepackage{amsmath}
\usepackage{amsthm}
\usepackage{array}
\usepackage{mathtools}
\usepackage{multirow}
\usepackage{tikz}
\usepackage{tikz-cd}
\usepackage{adjustbox}
\usepackage{nicefrac}
\usepackage{enumitem}
\usepackage{caption}
\usepackage{subcaption}

\usepackage[colorlinks=true, citecolor=blue]{hyperref}

\theoremstyle{plain}
\newtheorem{theorem}{Theorem}[section]
\newtheorem*{theorem*}{Theorem}
\newtheorem{cor}[theorem]{Corollary}
\newtheorem{lem}[theorem]{Lemma}
\newtheorem{prop}[theorem]{Proposition}

\newtheorem{question}{Question}

\theoremstyle{definition}
\newtheorem{defn}[theorem]{Definition}
\newtheorem{example}[theorem]{Example}

\theoremstyle{remark}
\newtheorem{rmk}[theorem]{Remark}

\let\phi=\varphi

\def\C{\mathbb{C}}
\def\Z{\mathbb{Z}}
\def\N{\mathbb{N}}

\def\O{\mathcal{O}}
\def\A{\mathcal{A}}
\def\P{\mathcal{P}}
\def\F{\mathcal{F}}

\let\ol\overline

\def\Res{\mathrm{Res}}
\def\Ind{\mathrm{Ind}}
\def\Hom{\mathrm{Hom}}
\def\ext{\mathrm{Ext}}

\def\g{\mathfrak{g}}

\def\s{\mathfrak{s}}

\def\n{\mathfrak{n}}
\def\h{\mathfrak{h}}
\def\sl{\mathfrak{sl}}
\def\gl{\mathfrak{gl}}

\def\b{\mathfrak{b}}
\def\sig{\sigma}

\def\eps{\varepsilon}
\def\del{\delta}
\def\px{\partial_\xi}
\def\Del{\Delta}
\def\part{\partial}
\def\lam{\lambda}
\def\pe{\mathfrak{pe}}

\DeclareMathOperator{\img}{im}

\title{Koszulity in the category $\O$ of the periplectic Lie superalgebra $\mathfrak{pe}(2)$}
\author{Finn Kinley}
\address{}
\email{finnkinley2000@gmail.com}

\begin{document}

\begin{abstract}
    The main result of this paper is to establish precisely which blocks in the category $\O$ of the periplectic Lie superalgebra $\pe(2)$ are Koszul. It is known that $\O(\pe(2))$ has three blocks up to equivalence; one generic block and two integral blocks. The generic block is known to be Koszul, and the principal integral block is verifiably not Koszul. In this paper, we prove that the remaining block of $\O(\pe(2))$ is Koszul. Consequently every non-principal block of $\O(\pe(2))$ is Koszul, which is the category $\O$ analogue, in the smallest case, of a known result for finite-dimensional modules over semisimple extensions of Takiff superalgebras. This is done by explicitly computing the endomorphism algebra of projective modules in this block and then proving that it is Koszul inductively. Along the way, we compute all $\ext$ groups between simples in this block. To compute the endomorphism algebra we are aided by a computer algebra tool developed in Mathematica, inspired by a post on Stack Exchange. 
\end{abstract}

\maketitle

\section{Introduction}

The periplectic Lie superalgebra $\pe(n)$ appears as one of the two ``strange'' superalgebras in the classification of simple Lie superalgebras \cite{KAC19778}. It is a subalgebra of $\mathfrak{gl}(n|n)$, preserving a specific non-degenerate odd symmetric bilinear form. Work has been done on studying the finite module category and category $\O$ of $\pe(n)$, for example 
\cite{KC_PBA}, \cite{CHEN201599}, \cite{periplect1}, \cite{Balagovi2016TranslationFA}, 
\cite{Chen_Jmiddle},
\cite{10.1093/imrn/rnae230},
\cite{Balagovi2018TheAV},
\cite{CHEN_MAZORCHUK_2023}, and \cite{SERGANOVA2002615}. 

In \cite[Theorem 1.1.3]{beilinson_ginzburg_soergel_1996} it was shown that the category $\O$ of any complex semisimple Lie algebra is Koszul. That is, the endomorphism algebra of its projective modules is a Koszul algebra. Koszul algebras are among the homologically best-behaved non-semisimple graded algebras. In contrast, to what extent Koszulity holds in the category $\O$ of Lie superalgebras is not well understood. It has been shown in \cite[Theorem A]{Brundan_Losev_Webster} that every integral block of the category $\O$ of $\mathfrak{gl}(m|n)$ is Koszul, and it is conjectured that the category $\O$ of $\mathfrak{osp}(m|2n)$ is Koszul (see \cite[Introduction, p.~14011]{10.1093/imrn/rnae230} and references therein), which gives hope for Koszulity of blocks in the category $\O$ of other Lie superalgebras. 

Existence of Koszul blocks in the category $\O$ of $\pe(n)$ is, for general $n$, not known. Some things can be said for $n=2$, thanks to \cite{periplect1} and \cite{KC-Tak}. Further, for arbitrary $n\geq2$, we know that not all blocks of $\O(\pe(n))$
(as classified in \cite{periplect1}) can be Koszul, as we can always find a
finite-dimensional module with
simple top and simple socle whose radical and socle filtrations do not
coincide \cite[\S 7.2]{10.1093/imrn/rnae230}.

The main goal of this paper is to establish precisely which blocks of $\O(\pe(2))$ are Koszul. The Lie superalgebra $\pe(2)$ sits at the intersection of two families: the periplectic Lie superalgebras $\pe(n)$, and the semisimple extensions of Takiff superalgebras, who are built from a finite-dimensional simple Lie algebra $\s$, with $\pe(2)$ the special case $\s=\sl_2$ \cite[\S 3.1.1]{KC-Tak}. For these extensions, the category $\F$ of finite-dimensional modules is studied extensively in \cite{KC-Tak}, and the blocks of both $\F$ and category $\O$ are classified. In particular, every non-principal block of $\F$ (one not containing the trivial module) is Koszul, for every such $\s$ \cite[Proposition 5.2.7 and \S 5.3.1]{KC-Tak}. In this paper we show that, in the smallest case $\s=\sl_2$, the same is true of category $\O$. Hence, two natural questions arise from the results of this paper, one for each family.

\begin{question}\label{q:periplectic}
    For $n>2$, which blocks of $\O(\pe(n))$ are Koszul?
\end{question}

\begin{question}\label{q:takiff}
    Let $\s$ be a finite-dimensional simple Lie algebra and let $\g_\s$ denote the semisimple extension of its Takiff superalgebra. Are the non-principal blocks of $\O(\g_\s)$ always Koszul?
\end{question}

The approach of this paper could in principle be applied to Question \ref{q:takiff} for other choices of $\s$, but the computations grow quickly. Roughly speaking, ours take place in modules which, as vector spaces, are of the form $\Lambda(\s)\otimes M$, where $M$ is a module in category $\O$ for $\s$ and $\Lambda(\s)$ is the exterior algebra of $\s$, of dimension $2^{\dim\s}$. This dimension is $8$ for $\s=\sl_2$ and already $256$ for $\sl_3$, the next smallest simple Lie algebra. The methods of this paper cannot be directly applied to Question \ref{q:periplectic}, since $\pe(n)$ is a semisimple extension of a Takiff superalgebra only for $n=2$ (by comparing the action of the centre of the even part on the odd part).

Specifically, this paper builds off work done in \cite{periplect1} and \cite{KC-Tak}, which individually show that $\O(\pe(2))$ has three blocks (up to equivalence), denoted by $\O_{\kappa}, \O_{odd}$ and $\O_{even}$. After computing the endomorphism algebra of the ``generic'' block $\O_{\kappa}$, it is easily seen to be Koszul  (see \cite[Theorem 6.5]{periplect1}). The principal block $\O_{even}$ is not Koszul (see \cite[\S 5.3.1]{KC-Tak}), as it contains a module with simple top and socle, but whose socle and radical filtrations do not coincide (such a module in a Koszul category must be \textit{rigid}, that is, these filtrations would have to coincide \cite[Proposition 2.4.1]{beilinson_ginzburg_soergel_1996}). In this paper, we prove that the remaining integral block $\O_{odd}$ is Koszul. In this case, computing the endomorphism algebra is quite involved, and Koszulity is verified with a finicky inductive proof. As part of the proof, we compute all $\ext$ groups between simples in $\O_{odd}$. To compute the endomorphism algebra we are aided by a computer algebra tool written in Mathematica, inspired by a post on Stack Exchange \cite{lamers_2018}.

\subsection*{Structure of the paper}

This paper is organised as follows. In Section \ref{sec:prelims} we explicitly define $\pe(2)$ and go over other preliminary definitions. In Section \ref{sec:indec_projectives} we describe the indecomposable projective modules in $\O_{odd}$ as inductions from projective modules in, what is essentially, $\O(\sl_2)$. In Section \ref{sec:describing_A}, we compute the dimensions of all homomorphism spaces between the projective modules in $\O(\pe(2))$, and then describe the endomorphism algebra of projectives as a quotient of a path algebra with quadratic relations. Finally, in Section \ref{sec:koszul} we prove that the endomorphism algebra of $\O_{odd}$ is Koszul and compute all higher order extension groups between simples in $\O_{odd}$. 

\subsection*{Acknowledgements}
 My greatest thanks to Kevin Coulembier, who offered me this problem for my Masters thesis, and whose continued support after finishing my degree facilitated the writing of this paper. 
 
 The writing of this paper was financially supported by the Trowbridge scholarship from The University of Sydney.

\section{Preliminaries}\label{sec:prelims}

For an overview and a general introduction to Lie superalgebras and their representations see \cite{Musson} and \cite{Serganova2017RepresentationsOL}. 

\subsection{Definition of \texorpdfstring{$\pe(2)$}{pe(2)}, Verma and simple modules}\label{sec:defns}
The periplectic Lie superalgebra of rank $n$ (see \cite[\S 1.1.5]{Cheng_Wang}) is defined as
\[
\pe(n)=\left\{\begin{pmatrix}A & B \\ C & -A^t \end{pmatrix} \ | \ B^t=B, \ C^t=-C \right\} \subseteq \gl(n|n).
\]

In the rank $n=2$ case, we can equivalently construct $\pe(2)$ as a semisimple extension of a Takiff superalgebra \cite[\S 3.1.1]{KC-Tak}. We work with this realisation because it is the setting of \cite{KC-Tak}, whose results we build on, and because it places $\pe(2)$ in the family of Lie superalgebras of Question \ref{q:takiff}.

We define explicitly this Takiff extension realisation of $\pe(2)$ in the following. Let $\s=\sl_2$. We fix a triangular decomposition $\s=\n_\s^-\oplus \h_\s\oplus \n_\s^+$ where 
\begin{align*}
    y := \begin{bmatrix} 0 & 0 \\ 1 & 0 \end{bmatrix}, \ \ h:=\begin{bmatrix} 1 & 0 \\ 0 & -1 \end{bmatrix} \ \ \mathrm{ and }  \ \ x:=\begin{bmatrix} 0 & 1 \\ 0 & 0 \end{bmatrix} 
\end{align*}
span $\n_\s^-,\h_\s$ and $\n_\s^+$ respectively. Define the vector space \[\g=(\s\otimes\C\xi)\oplus (\s\oplus \C\xi\px)\oplus (\C\px),\] where $\xi,\xi\px$ and $\px$ are formal variables. We write $\g_{-1}=\s\otimes\C\xi$, $\g_0=\s\oplus\C\xi\px$, $\g_1=\C\px$, and give $\g$ a $\Z_2$-grading by letting \[\g_{\ol{0}}=\g_0  \ \mathrm{ and } \ \g_{\ol{1}}=\g_{-1}\oplus\g_1.\] To endow $\g$ with the structure of a Lie superalgebra, we define a superbracket $[\cdot,\cdot]:\g\times\g\to\g$ on basis elements of $\g$. We let this bracket, when restricted to $\s$, be the usual bracket on $\sl_2$. Now let $f,g\in\s$ and define the following bracket relations.

\begin{enumerate}[label=(\alph*)]
    \item $[\g_1,\g_1]=0=[\g_{-1},\g_{-1}]$.
    \item $[f,\xi\px]=0=[\xi\px,\xi\px]$.
    \item $[\px,f\otimes\xi]=f$.
    \item $[f,g\otimes\xi]=[f,g]\otimes\xi$.
    \item $[f,\px]=0$.
    \item $[\xi\px,f\otimes\xi]=f\otimes\xi$.
    \item $[\px,\xi\px]=\px$.
\end{enumerate}

As remarked in \cite[\S 3.1.1]{KC-Tak}, $\g\cong\pe(2)$. Explicitly, writing $E_{ij}$ for the elementary matrices in $\gl(2|2)$, an isomorphism is given by
\begin{align*}
    x&\mapsto E_{12}-E_{43}, & x\otimes\xi&\mapsto E_{13},\\
    y&\mapsto E_{21}-E_{34}, & y\otimes\xi&\mapsto -E_{24},\\
    h&\mapsto E_{11}-E_{22}-E_{33}+E_{44}, & h\otimes\xi&\mapsto -E_{14}-E_{23},\\
    \xi\px&\mapsto \tfrac{1}{2}(E_{11}+E_{22}-E_{33}-E_{44}), & \px&\mapsto E_{32}-E_{41}.
\end{align*}
We will use this description of $\pe(2)$ for the remainder of the paper.

We fix a triangular decomposition $\g=\n^-\oplus\h\oplus\n^+$ by setting 
\[\n^-=\g_{-1}\oplus \n_\s^-, \ \ \h=\h_\s\oplus\C\xi\px \ \mathrm{ and }\ \n^+=\n^+_\s\oplus\g_1.\]
Further, we have a triangular decomposition $\g_0=\n^-_0\oplus \h_0\oplus \n^+_0$, where $\n_0^-=\n_s^-$, $\h_0=\h$ and $\n_0^+=\n_s^+$.

Weights $\lam\in\h^*\cong\C^2$ are written as $\lam=a\eps+b\del$ where $\eps,\del\in\h^*$ are uniquely defined by \[\del(h)=0, \   \del(\xi\px)=-1\ \text{ and } \  \eps(h)=1,\ \eps(\xi\px)=0. \]
Using this basis of $\h^*$, the sets of positive and negative roots ($\Phi_+$ and $\Phi_-$ respectively) are given in Figure \ref{fig:roots}.

\begin{figure}[H]
    \centering
    \begin{subfigure}{.5\textwidth}
        \begin{tabular}{c|c}
          Basis of $\n^+$ & Roots in $\Phi_+$  \\
            \hline
            $x$ & $2\eps$\\
            $\px$ & $\delta $\\
        \end{tabular}
    \label{fig:roots+}
    \end{subfigure}%
    \begin{subfigure}{.5\textwidth}
            \begin{tabular}{c|c}
          Basis of $\n^-$ & Roots in $\Phi_-$  \\
            \hline
          $y$ & $-2\eps$ \\
          $y\otimes \xi$ & $-2\eps -\del$ \\
          $h\otimes \xi$ & $-\del$ \\
          $x\otimes\xi$ & $2\eps -\del$
        \end{tabular}  
    \label{fig:roots-}   
    \end{subfigure}
    \caption{Roots in $\Phi_+$ and $\Phi_-$ for basis elements of $\g$. }
    \label{fig:roots}
    \end{figure}

We write $\b=\h\oplus\n^+$ and call it the \textit{Borel} subalgebra of $\g$. Denote by $U(\g)$ the \textit{universal enveloping algebra} of $\g$ (as defined in \cite[\S 6.1]{Musson}, for example).

Define the \textit{Verma module} of highest weight $\lam$ as \[\Del(\lam):=\Ind_\b^\g\C_\lam=U(\g)\otimes_{U(\b)}\C_\lam,\] where $\C_\lam$ is the $\b$-module spanned by $w$ with action $\n^+\cdot w=0 \ \text{ and } \ hw=\lam(h)w \ \ \forall h\in\h$. The vector $1\otimes w\in U(\g)\otimes_{U(\b)}\C_\lam$ is the generator of $\Del(\lam)$. Now let $\b_0=\h_0\oplus\n^+_0$ be the Borel subalgebra of $\g_0$; then we can similarly define the $\g_0$-Verma module $\Del_0(\lam)$ of highest weight $\lam$. We call $
\Ind_\b^\g\C_\lam$ the \textit{Borel} presentation of $\Del(\lam)$. A second presentation of $\Del(\lam)$ arises as \[\Del(\lam)\cong\Ind_{\g_{\geq0}}^\g\Del_0(\lam),\] where we think of $\Del_0(\lam)$ as a $\g_{\geq0}$-module with trivial $\g_1$-action, and the isomorphism follows from the PBW theorem (see \cite[Theorem 6.1.1]{Musson}).

The simple modules $L(\lam)$ and $L_0(\lam)$ are the simple quotients of $\Del(\lam)$ and $\Del_0(\lam)$ respectively.

To simplify notation, let \[x_-=x\otimes\xi, \  y_-=y\otimes\xi  \ \  \mathrm{ and } \ \ h_-=h\otimes\xi.\] Then the PBW theorem gives that $\Del(\lam)\cong U(\n^-)\otimes \C_\lam$ as vector spaces. In particular, 
\begin{equation}\label{eq:PBW_basis}
    \{y^ny_-^{\alpha_1}h_-^{\alpha_2}x_-^{\alpha_3}\otimes w \ | \ n\in\N, \ \alpha_i=0,1\}
\end{equation}
 forms a basis of $\Del(\lam)$ (where the monomials are written in the order shown), written in the Borel presentation (call this the \textit{PBW basis} of $\Del(\lam)$).

Denote by $\O(\g)$ and $\O(\g_0)$ the category $\O$ of $\g$ and $\g_0$ respectively (denoted $\tilde{\O}$ and $\O$ in \cite[\S8.2.3]{Musson}); that is, $\O(\g)$ consists of those graded $U(\g)$-modules which lie in $\O(\g_0)$ when regarded as $U(\g_0)$-modules by restriction. Moreover, it is easy to check that $\Del(\lam)\in\O(\g)$ and $\Del_0(\lam)\in\O(\g_0)$.

\begin{rmk}\label{rmk:g0_sl2}
Let $\O(\g_0)_b$ be the full subcategory of $\O(\g_0)$ where weights have $\delta$ coefficient equal to $b$. Then clearly $\O(\g_0)_b\cong\O(\sl_2)$ (which we understand very well, see \cite[Chapter 5]{Mazorchuk} for instance) and $\O(\g_0)=\bigoplus_b\O(\g_0)_b$.
\end{rmk}

We will describe the projective modules in the block of interest in $\O(\g)$ as inductions from the projective modules of $\O(\g_0)$, which we understand well by Remark \ref{rmk:g0_sl2}.

\subsection{Reductions, decompositions and blocks of \texorpdfstring{$\O(\g)$}{O(g)}}

By the same weight considerations given in \cite[\S 4.1.1]{KC-Tak} on reducing the problem of studying finite dimensional $\g$-modules, we have a decomposition \[\O(\g)\cong\bigoplus_{x\in\C/\Z}\O_{x},\]
where $\O_{x}$ is the full subcategory of $\O(\g)$ containing modules all of whose weights are of the form $a\eps+b\del$ with $b\in x$. This decomposition follows from the observation that elements of $\g$ change the $\xi\px$-component of a weight by only $0$ or $\pm 1$. Let $\Lambda'\subset\h^*$ be the set of weights of the form $a\eps+b\del$ with $b\in\Z$, and denote by $\O_{\Lambda '}$ the full subcategory of $\O(\g)$ consisting of those modules all of whose weights lie in $\Lambda'$. Then $\O_{x}\cong\O_{\Lambda '}$ for all $x\in\C/\Z$, so we reduce the problem of studying $\O(\g)$ to studying $\O_{\Lambda'}$.

As a final reduction (also analogously seen in \cite[\S 4.1.2]{KC-Tak}), we can decompose $\O_{\Lambda'}=\O_0\oplus \O_1$, where weight vectors in a module $M\in \O_0$ of weight $a\eps+b\del \in\Lambda'$ have degree $b\mod 2$ and $\O_1$ is the image of $\O_0$ under the functor $\Pi:\O\to\O$ which swaps parity. Hence, from here on $\O$ will refer to $\O_0$ without loss of generality. 

Let $\Lambda\subset\Lambda '$ be the set of weights $a\eps+b\del$ such that $a,b\in\Z$. Then \cite[Proposition 4.3.7]{KC-Tak} describes the block decomposition of $\O_\Lambda$ (the full subcategory of $\O$ consisting of those modules all of whose weights lie in $\Lambda$).

\begin{theorem}\label{thm:blocks}
    We have that $\O_\Lambda$ decomposes into three blocks. The highest weights of the simple modules in each block are
    \begin{align*}
        &\{2n\eps + b\del \ | \ n,b\in\Z\},\\
        &\{(2n+1)\eps + (2b-n)\del \ | \ n,b\in\Z\},\\
        &\{(2n+1)\eps + (2b-n-1)\del \ | \ n,b\in\Z\}.
    \end{align*}
    Denote these three blocks by $\O_{even},\O_{odd}^0$ and $\O_{odd}^1$ respectively. 
\end{theorem}
\begin{proof}
    See \cite[Proposition 4.3.7]{KC-Tak}. 
\end{proof}

Finally, we provide the block decomposition of $\O$ seen in \cite[Theorem 5.4]{periplect1}. Let $K=(\C/2\Z)\setminus\{[0],[1]\}$ be the set of cosets in $\C/2\Z$ excluding the even and odd integer cosets. Let $\kappa\in K$ and define $\O_\kappa$ to be the full subcategory of $\O$ consisting of all modules whose weights are of the form $a\eps+b\del$ where $a\in\kappa$ and $b\in\Z$. Then

\begin{equation}\label{eq:full_block_decomp}
        \O=\bigoplus_{\kappa\in K}\O_\kappa\oplus \O_{even}\oplus \O_{odd}^0\oplus \O_{odd}^1,
\end{equation}
is the block decomposition of $\O$. It was also shown in \cite[Theorem 6.5]{periplect1} that $\O_{odd}^0\cong \O_{odd}^1$, $\O_{\kappa_1}\cong\O_{\kappa_2}$ for all $\kappa_1,\kappa_2\in K$, and that these are the only equivalences between blocks. Henceforth, we write $\O_{odd}$ for $\O_{odd}^0$. Finally, we denote by $\Lambda_{odd}$ the set of highest weights of the simple modules in $\O_{odd}$.

\section{Indecomposable projectives in \texorpdfstring{$\O_{odd}$}{Oodd}}\label{sec:indec_projectives}
First observe that in $U(\g)$, $\px^2+\px^2=[\px,\px]=0$ and so $\px^2=0$. In particular, for the linear map $\px:M\to M$ on any $\g$-module $M$, we have $\img\px\subseteq\ker\px$. In $\O_{odd}$ and $\O_{\kappa}$, we show that the other inclusion holds and that we can write indecomposable projective modules as inductions from indecomposable projectives in $\O(\g_0)$. 

For the following, we let $\lam=a\eps+b\del\in\h^*$. We write $\ker\px(M)$ and $\img\px(M)$ where $M\in\O$ for the image and kernel of the linear map $\px:M\to M$. Further, let $P_0(\lam)\in\O(\g_0)$ be the indecomposable projective cover of $L_0(\lam)\in\O(\g_0)$. As $\sl_2$-modules, $P_0(\lam)$ is isomorphic to the indecomposable projective cover of the simple $\sl_2$-module of highest weight $a\eps$, as per Remark \ref{rmk:g0_sl2}. 
\begin{lem}\label{lem:kerpx=impx_simples}
    If $a\neq 0$ then $\ker\px(L(\lam))= \img\px(L(\lam))$. 
\end{lem}
\begin{proof}
    We have $L(\lam)\cong L_0(\lam)\oplus L_0(\lam-\del)$ as $\g_0$-modules by the same argument as in \cite[Lemma 4.2.1]{KC-Tak}. Let $v_{\lam-\del}$ and $v_\lam$ be maximal vectors of $L_0(\lam-\del)$ and $L_0(\lam)$ in $L(
    \lam)$ respectively. Clearly, $\px v_{\lam-\del}\in L(\lam)_\lam$, so $\px v_{\lam-\del}=\alpha v_\lam$ for some $\alpha\in\C$. If $\alpha=0$ then $v_{\lam-\del}$ would be highest weight in $L(\lam)$, which is nonsense. Hence, $\alpha\neq 0$ and the $\g_0$-module homomorphism $\px |_{L_0(\lam-\del)}:L_0(\lam-\del)\to L_0(\lam)$ is bijective. Finally, because $L_0(\lam)=U(\g_0)v_\lam$, $\px$ commutes with $\s$ and $[\px,\xi\px]=\px$, we have \[\ker\px(L(\lam))=L_0(\lam)=\img\px(L(\lam)).\] 
\end{proof}

\begin{rmk}
    The constructions and results of this section apply equally to $\O_\kappa$, since the highest weight $a\eps+b\del$ of every simple module in $\O_\kappa$ satisfies $a\neq0$, as required in Lemma \ref{lem:kerpx=impx_simples}. As the endomorphism algebra of $\O_\kappa$ is already known to be Koszul \cite[Theorem 6.5]{periplect1}, we restrict our attention to $\O_{odd}$ henceforth.
\end{rmk}

\begin{rmk}
    This same argument does not work for modules in $\O_{even}$ because when $a=0$ (so $\lam=b\del$), one can show that $L(\lam)\cong L_0(\lam)$ as $\g_0$-modules. 
\end{rmk}

\begin{cor}\label{cor:kerpx=imgpx_all_O}
    Let $M\in\O_{odd}$. Then $\ker\px(M)=\img\px(M)$. 
\end{cor}
\begin{proof}
    We do induction on the length $n$ of $M\in\O_{odd}$. If $n=1$, then Lemma \ref{lem:kerpx=impx_simples} immediately gives the result. Recall that $\img\px(M)\subseteq\ker\px(M)$ is always true, so we only need to show the other inclusion. The result follows by induction and a standard diagram chase on
    $
        \begin{tikzcd}
            0 \arrow[r] & N \arrow[r] & M \arrow[r] & M/N \arrow[r] & 0,
        \end{tikzcd}
    $
    where $N\subset M$ is a maximal submodule and $M/N$ is its simple quotient.
\end{proof}

Since $\g_0$ commutes with $\px$, it follows that $\ker\px(M)$ is a $\g_0$-submodule of $M\in\O$ (upon restriction). Furthermore, because $M\in\O(\g_0)$ by restriction, we have $\ker\px(M)\in\O(\g_0)$. In this way, 
we can think of $\ker\px$ as a (covariant) functor from $\O$ to $\O(\g_0)$, where a morphism $f:M\to M'$ in $\O$ is associated to $\hat{f}=f|_{\ker\px(M)}$. 

\begin{lem}\label{lem:exactnessofkerpx}
    When restricted to modules in $\O_{odd}$, the functor $\ker\px$ is exact.
\end{lem}
\begin{proof}
    Let \[
    \begin{tikzcd}
        0\arrow[r] & M' \arrow[r, "f"] & M \arrow[r, "g"] & M'' \arrow[r] & 0
    \end{tikzcd}
    \]
    be a short exact sequence of modules in $\O_{odd}$. We claim that after applying $\ker\px$ to the above short exact sequence, it remains exact. Exactness at $\ker\px(M')$ and $\ker\px(M)$ holds for any such short exact sequence in $\O$. Exactness at $\ker\px(M'')$ follows from $\ker\px (M'')=\img\px (M'')$ (by Corollary \ref{cor:kerpx=imgpx_all_O}) and the surjectivity of $g$.
\end{proof}

Because we can think of modules in $\O(\g_0)$ as $U(\g_{\geq0})$-modules by letting $\g_1$ act as zero, we have

\begin{prop}\label{prop:projectives}
    Let $\lam\in\Lambda_{odd}$. Then, \[P(\lam):=\Ind_{\g_{\geq 0}}^\g P_0(\lam)\] is a projective $\g$-module. 
\end{prop}
\begin{proof}
    We have the adjunction $\Hom_\g(P(\lam),-)\cong\Hom_{\g_{\geq 0}}(P_0(\lam),\Res^\g_{\g_{\geq 0}}-)$. Observe that for any \(\phi\in\Hom_{\g_{\geq 0}}(P_0(\lam),\allowbreak \Res_{\g_{\geq 0}}^\g M)\), we have $\px\phi(v)=\phi(\px v)=0$ and so as functors
    \begin{align*}
        \Hom_{\g_{\geq 0}}(P_0(\lam),\Res^\g_{\g_{\geq 0}}-)&=\Hom_{\g_{\geq 0}}(P_0(\lam),\ker\px-)\\
        &= \Hom_{\g_0}(P_0(\lam), \ker\px -),
    \end{align*}
     which is exact in $\O_{odd}$ as it is a composition of exact functors. 
\end{proof}

\begin{rmk}\label{rmk:gen_of_projs}
    Let $\lam = a\eps+b\del\in\Lambda_{odd}$. By Remark \ref{rmk:g0_sl2} and \cite[\S 5.2]{Mazorchuk}, the module $P(\lam)$ is generated by the single element $1\otimes u_0\in\Ind_{\g_{\geq0}}^\g P_0(\lam)$ of weight $\lam$, where $u_0\in P_0(\lam)_\lam$ generates $P_0(\lam)$. When $a\geq-1$ we have $P_0(\lam)=\Del_0(\lam)$, and we identify $1\otimes u_0$ with the generator $1\otimes w\in\Ind_\b^\g\C_\lam$ of $\Del(\lam)$ in the Borel presentation. When $a\leq-3$, the module $P_0(\lam)$ is a proper extension of $\Del_0(\lam)$; for the specific basis used and visualisation of the structure see Figure \ref{fig:vecs_to_consider_in_P0} in Section \ref{sec:targets_non_verma}. The specific projective module $P_0(\lam)$ that $u_0$ belongs to will always be clear by context.
\end{rmk}

\begin{prop}
    The projective module $P(\lam)$ is the indecomposable projective cover of $L(\lam)$ in $\O_{odd}$.
\end{prop}
\begin{proof}
    As in the proof of Proposition \ref{prop:projectives}, we have \[\Hom_\g(P(\lam),L(\mu))\cong \Hom_{\g_0}(P_0(\lam), \ker\px(L(\mu))).\]
    But $\ker\px(L(\mu))\cong L_0(\mu)$ as seen in the proof of Lemma \ref{lem:kerpx=impx_simples}. Hence, \[\Hom_{\g_0}(P_0(\lam), \ker\px( L(\mu)))\cong \Hom_{\g_0}(P_0(\lam), L_0(\mu))\cong\C\delta_{\lam\mu}.\]

\end{proof}

\begin{rmk}\label{rmk:Odd_projs}
    Let $\lam = a\eps+b\del\in\Lambda_{odd}$. We now know that $P(\lam)=\Ind_{\g_{\geq0}}^\g P_0(\lam)$ and $\Del(\lam)=\Ind_{\g_{\geq 0}}^\g\Del_0(\lam)$. Moreover, from the theory of $\O(\sl_2)$ we have $P_0(\lam)=\Del_0(\lam)$ if and only if $a\not\in\{-2,-3,-4,\dots\}$. Hence, $P(\lam)=\Del(\lam)$ if and only if $a\notin\{-2,-3,-4,\dots\}$.
\end{rmk} 

\begin{defn}
    Let $\lam = a\eps+b\del$ be a weight of $\g$. The \textit{dual} of $\lam$ is the weight  $\lam':=(-a-2)\eps+b\del$. For a module $M\in \O$ with weight space $M_\lam$, its \textit{dual weight space} is $M_{\lam'}$.
\end{defn}

\begin{lem}\label{lem:SES_projective_verma}
    Let $\lam=a\eps+b\del\in\Lambda_{odd}$ with $a\leq -3$. We have a short exact sequence 
    \[\begin{tikzcd}
        0\arrow[r] &\Del(\lam') \arrow[r] & P(\lam)\arrow[r] & \Del(
        \lam)\arrow[r] & 0.
    \end{tikzcd}\]
\end{lem}
\begin{proof}
In $\O(\g_0)$, we have the short exact sequence
\[\begin{tikzcd}
        0\arrow[r] &\Del_0(\lam') \arrow[r] & P_0(\lam)\arrow[r] & \Del_0(
        \lam)\arrow[r] & 0.
    \end{tikzcd}\]
    Applying the functor $U(\g)\otimes_{U(\g_{\geq 0})}-$ to this exact sequence yields the desired exact sequence, since $U(\g)$ is a flat $U(\g_{\geq 0})$-module by the PBW theorem. 
\end{proof}

\section{The endomorphism algebra of \texorpdfstring{$\O_{odd}$}{Oodd}}\label{sec:describing_A}
We now describe the endomorphism algebra \[\A = \bigoplus_{\mu,\lam\in\Lambda_{odd}}\Hom_\g(P(\mu),P(\lam)).\] We will show that $\A\cong \C Q/I$, where $\C Q$ is the path algebra of the quiver $Q$ consisting of a degree one generating set of $\A$, and $I$ an ideal of quadratic relations. In the following, we let $\lam,\mu\in\Lambda_{odd}$ unless otherwise stated (so we only consider coefficients of $\eps$ that are odd integers).

\subsection{Computing dimensions}
We start by computing $\dim\Hom_\g(P(\mu),P(\lam))=[P(\lam):L(\mu)]$. The following propositions will reduce this problem to one of studying the multiplicities $[\Del_0(\lam):L_0(\mu)]$. 

\begin{prop}\label{prop:dim_step3}
    Let $\lam=a\eps+b\delta\in\Lambda_{odd}$. Then we have the following cases for $[P(\lam):L(\mu)]$.
    \begin{enumerate}[label=(\roman*)]
        \item If $a\geq -1$ then $[P(\lam):L(\mu)]=[\Del(\lam):L(\mu)]$.
        \item If $a\leq -3$ then \[[P(\lam):L(\mu)]=[\Del(\lam):L(\mu)]+[\Del(\lam'):L(\mu)].\] 
    \end{enumerate}
\end{prop}
\begin{proof}
    Part $(i)$ is because $P(\lam)=\Del(\lam)$ in this case (as seen in Remark \ref{rmk:gen_of_projs}). Part $(ii)$ follows from Lemma \ref{lem:SES_projective_verma}.
\end{proof}

\begin{lem}\label{lem:ses_lem}
    We have 
    \[
    [V:L_0(\mu)]=\sum_\nu[V:L(\nu)][L(\nu):L_0(\mu)]  
    \] 
    for any $V\in\O$. 
\end{lem}
\begin{proof}
    Take a composition series of $V\in\O(\g)$ and then refine it to a composition series in $\O(\g_0)$. 
\end{proof}

\begin{prop}\label{prop:dim_step2}
    Let $\lam\in\h^*$. If $\mu=a\eps+b\del$ where $a\neq 0$, then \[[\Del(\lam):L(\mu)]+[\Del(\lam):L(\mu+\del)]=[\Del(\lam):L_0(\mu)].
    \]
\end{prop}
\begin{proof}
    As in the proof of Lemma \ref{lem:kerpx=impx_simples}, $L(\nu)\cong L_0(\nu)\oplus L_0(\nu-\del)$ as $\g_0$-modules and so $[L(\nu):L_0(\mu)]=1$ if $\nu=\mu,\mu+\del$ and is $0$ otherwise. The proposition now follows by Lemma \ref{lem:ses_lem}. 
\end{proof}

\begin{lem}\label{lem:technical_for_last_prop_dims}
    Let $M,N\in\O$. If $\dim M_\lam=\dim N_\lam$ for all $\lam\in\h^*$, then $[M:L]=[N:L]$ for all simple modules $L\in\O$. 
\end{lem}
\begin{proof}
    A standard result in category $\O$.
\end{proof}

Let $\lam=a\eps+b\delta$ for $a,b\in\Z$ and define $\lam_{ijk}=\lam+i(2\eps-\delta)-j\delta-k(2\eps+\delta)$ for $i,j,k=0,1$. Note that $\lam_{ijk}$ is the weight of $y_-^kh_-^jx_-^iv_\lam$, where $v_\lam$ has weight $\lam$.

\begin{prop}\label{prop:dim_step1}
    We have \[ [\Del(\lam):L_0(\mu)]=\sum_{i,j,k=0}^1[\Del_0(\lam_{ijk}):L_0(\mu)].\]
\end{prop}
\begin{proof}
   It is straightforward to check that the weight spaces of $\Del(\lam)$ and $\bigoplus_{i,j,k=0}^1\Del_0(\lam_{ijk})$ have the same dimension. The result then follows from Lemma \ref{lem:technical_for_last_prop_dims}.
\end{proof}

We now begin the computation of $\dim\Hom_\g(P(\mu),P(\lam))=[P(\lam):L(\mu)]$. We will start by computing $[\Del(\lam):L_0(\mu)]$ using Proposition \ref{prop:dim_step1}, which is really an exercise in $\O(\sl_2)$ as per Remark \ref{rmk:g0_sl2}. Then we apply Propositions \ref{prop:dim_step2} and \ref{prop:dim_step3} to these numbers to arrive at $[P(\lam):L(\mu)]$. 

When $a\geq 3$, the coefficient of $\eps$ in $\lam_{ijk}$ is always non-negative. For $a\leq-3$, the same coefficient is always negative. Moreover, we have the following composition series in $\O(\g_0)$.

\begin{enumerate}[label=(\roman*)]
    \item If $a\in\N$: $\Delta_0(a\eps+b\delta)\supset L_0((-a-2)\eps+b\delta)\supset 0$ is a composition series.
    \item If $a\not\in\N$: $\Delta_0(a\eps+b\delta)\supset 0$ is a composition series. 
\end{enumerate}

 By considering the first composition series and employing Proposition \ref{prop:dim_step1}, we construct Table \ref{tab:comp-srs1} which enumerates the choices of $\mu$ that make $[\Delta(\lam):L_0(\mu)]$ non-zero for $a\geq3$. Note that the choices of $\mu=\lam_{ijk}$ always work regardless of the sign of $a$; we call such choices \textit{trivial}. In the case where a trivial choice has coefficient of $\eps$ non-negative, we can also choose the dual of this trivial choice. Using Table \ref{tab:comp-srs1}, we see that $[\Delta(\lam):L_0(\mu)]\leq 1$ when $a\geq 3$. 
\begin{table}[ht]
    \centering
    \begin{tabular}{c|c|c}
      $(i,j,k)$ & trivial: $\mu = \lam_{ijk}$ & dual: $\mu = \lam_{ijk}'$ \\
      \hline
      $(0,0,0)$ & $a\eps+b\delta$          & $(-a-2)\eps+b\delta$\\
      $(1,0,0)$ & $(a+2)\eps+(b-1)\del$     & $(-a-4)\eps+(b-1)\del$\\
      $(0,1,0)$ & $a\eps+(b-1)\del$         & $(-a-2)\eps+(b-1)\del$\\
      $(0,0,1)$ & $(a-2)\eps+(b-1)\del$     & $-a\eps +(b-1)\del$\\
      $(1,1,0)$ & $(a+2)\eps+(b-2)\del$     & $(-a-4)\eps+(b-2)\del$\\
      $(1,0,1)$ & $a\eps+(b-2)\del$         & $(-a-2)\eps+(b-2)\del$\\
      $(0,1,1)$ & $(a-2)\eps + (b-2)\del$   & $-a\eps+(b-2)\del$\\
      $(1,1,1)$ & $a\eps+(b-3)\del$         & $(-a-2)\eps+(b-3)\del$
    \end{tabular}
    \caption{Choices of $\mu$ that make $[\Delta(\lam):L_0(\mu)]$ non-zero for $a\geq 3$. }
    \label{tab:comp-srs1}
\end{table}

If $a\leq-3$, then the coefficient of $\eps$ in $\lam_{ijk}$ will always be strictly negative. Hence, only the trivial choices $\mu=\lam_{ijk}$ make $[\Delta(\lam):L_0(\mu)]$ non-zero. Combining these observations gives $[\Delta(\lam):L_0(\mu)]\leq 1$ for $a\neq -1,1$, with equality for appropriate choices of $\mu$ outlined above. 

Similar computations to those used to get Table \ref{tab:comp-srs1} can be done for the edge cases $a=-1,1$, giving Table \ref{tab:comp-edge}. Note that dual choices are omitted if the coefficient of $\eps$ in the corresponding trivial choice is always negative.

\begin{table}[ht]
    \centering
    \begin{tabular}{c|c|c}
      $(i,j,k)$ & trivial: $\mu=\lam_{ijk}$ & dual: $\mu=\lam_{ijk}'$ \\
      \hline
      $(0,0,0)$ & $a\eps+b\delta$          & $(-a-2)\eps+b\delta \ *$\\
      $(1,0,0)$ & $(a+2)\eps+(b-1)\del$     & $(-a-4)\eps+(b-1)\del$\\
      $(0,1,0)$ & $a\eps+(b-1)\del$         & $(-a-2)\eps+(b-1)\del \ *$\\
      $(0,0,1)$ & $(a-2)\eps+(b-1)\del$     & - \\
      $(1,1,0)$ & $(a+2)\eps+(b-2)\del$     & $(-a-4)\eps+(b-2)\del$\\
      $(1,0,1)$ & $a\eps+(b-2)\del$         & $(-a-2)\eps+(b-2)\del \ *$\\
      $(0,1,1)$ & $(a-2)\eps + (b-2)\del$   & - \\
      $(1,1,1)$ & $a\eps+(b-3)\del$         & $(-a-2)\eps+(b-3)\del \ *$
    \end{tabular}
    \caption{Choices of $\mu$ that make $[\Delta(\lam):L_0(\mu)]$ non-zero for $a=-1,1$. The $*$ choices are only valid if $a=1$.}
    \label{tab:comp-edge}
\end{table}

Note that the trivial and dual (when they exist) choices of rows $(0,0,0)$ and $(1,1,1)$ in Table \ref{tab:comp-edge} will always give $[\Delta(\lam):L_0(\mu)]= 1$ by considering the coefficient of $\delta$. When $a=-1,1$, all choices in Table \ref{tab:comp-edge} give $[\Delta(\lam):L_0(\mu)]= 1$, except if $a=-1$, in which case taking \[\mu=-3\eps+(b-1)\delta \text{ or } \mu=-3\eps+(b-2)\delta\] makes $[\Delta(\lam):L_0(\mu)]=2$.

\begin{prop}\label{prop:dim_counts}
    Let $\lam=a\eps+b\del\in\Lambda_{odd}$. The composition multiplicities $[P(\lam):L(\mu)]$ are given as follows. 
    \begin{enumerate}
        \item If $a\geq 3$: \[
        [P(\lam):L(\mu)]=
        \begin{cases}
            1 & \text{if } \mu=a_e\eps+(b-n)\del, \  n=0,2 \text{ or }  \mu=a_o\eps+(b-1)\del,\\
            0 & \text{otherwise. }
        \end{cases}\]
        Here $a_e\in\{a,-a-2\}$ and $a_o\in\{a+2,a-2,-a-4,-a\}$.
        \item If $a=1$: The composition multiplicity is identical to the $a\geq 3$ case, but we only allow $a_o\in\{a+2,a-2,-a-4\}$.
        \item If $a=-1$: 
        \[
            [P(\lam):L(\mu)]=
        \begin{cases}
            2 & \text{if } \mu=-3\eps+(b-1)\del,\\
            1 & \text{if } \mu=-\eps+(b-n)\del, \  n=0,2 \text{ or }  \mu=\eps+(b-1)\del,\\
            0 & \text{otherwise. }
        \end{cases}
        \]
        Note that $-3=a-2=-a-4$. 
        \item If $a=-3$: The composition multiplicities are identical to the $a\leq -5$ case, but we require $a_o=-a$.
        \item If $a\leq -5$:  \[
        [P(\lam):L(\mu)]=
        \begin{cases}
            2 & \text{if } \mu=a\eps+(b-n)\del, \  n=0,2 \text{ or } \mu=(a\pm2)\eps+(b-1)\del,\\
            1 & \text{if } \ \mu=(-a-2)\eps+(b-n)\del,\   n=0,2  \text{ or } \mu=a_o\eps+(b-1)\del,\\
            0 & \text{otherwise.}
        \end{cases}\]
        Here $a_o\in\{-a,-a-4\}$. 
    \end{enumerate}
\end{prop}
\begin{proof}
    Proposition \ref{prop:dim_step3} tells us that $[P(\lam):L(\mu)]=[\Del(\lam):L(\mu)]+[\Del(\lam'):L(\mu)]$. We can compute $[\Del(\lam):L(\mu)]$ and $[\Del(\lam'):L(\mu)]$ using Proposition \ref{prop:dim_step2} and our descriptions of $[\Del(\lam):L_0(\mu)]$.
    
    To illustrate the computation, take the example $\lam=-3\eps+b\del$. From our previous discussion, $[\Del(\lam):L_0(\mu)]$ is equal to one for $\mu=\lam_{ijk}$ and zero otherwise. 
\begin{table}[ht]
    \centering
    \begin{tabular}{c|c}
      $(i,j,k)$ & $\lam_{ijk}$  \\
      \hline
      $(0,0,0)$ & $-3\eps+b\delta$         \\
      $(1,0,0)$ & $-\eps+(b-1)\del$     \\
      $(0,1,0)$ & $-3\eps+(b-1)\del$        \\
      $(0,0,1)$ & $-5\eps+(b-1)\del$     \\
      $(1,1,0)$ & $-\eps+(b-2)\del$     \\
      $(1,0,1)$ & $-3\eps+(b-2)\del$         \\
      $(0,1,1)$ & $-5\eps + (b-2)\del$    \\
      $(1,1,1)$ & $-3\eps+(b-3)\del$
    \end{tabular}
    \caption{The $\lam_{ijk}$ in the case $a=-3$.}
    \label{tab:comp_a=-3}
\end{table}
    Consider weights in Table \ref{tab:comp_a=-3} of the form $\mu_n=-3\eps+(b-n)\del$, where $n=0,1,2,3$. We know that $c\eps+(b-4)\del$ is not a weight in $\Del(\lam)$ for any $c\in\C$. This means $[\Del(\lam):L(\mu_3-\del)]=0$. Hence, applying Proposition \ref{prop:dim_step2} with $\mu=\mu_3-\del$ and considering Table \ref{tab:comp_a=-3} gives $[\Del(\lam):L(\mu_3)]=0$. Using this, we can apply Proposition \ref{prop:dim_step2} again to get $[\Del(\lam):L(\mu_2)]+0=1$. By iterating this process two more times, one obtains that $[\Del(\lam):L(\mu_n)]=1$ if $n=0,2$ and zero otherwise. By applying this procedure to $\mu=-\eps+(b-m)\del$ and $\mu=-5\eps+(b-m)\del$ for $m=1,2$, we get $[\Del(\lam):L(\mu)]=1$ when \[\mu=-3\eps+(b-n)\del, \ -\eps+(b-1)\del, \ -5\eps+(b-1)\del\] with $n=0,2$, and $[\Del(\lam):L(\mu)]=0$ otherwise. 
   
    The same process can be used to compute $[\Del(\lam'):L(\mu)]=1$ if $\mu$ is taken to be any of the following weights,
    \begin{align*}
        &\eps+(b-n)\del, \ \ -3\eps+(b-n)\del,\\
        -5\eps+&(b-1)\del, \ \ -\eps+(b-1)\del, \ \ 3\eps+(b-1)\del,
    \end{align*}
    where $n=0,2$, and $[\Del(\lam'):L(\mu)]=0$ otherwise.
\end{proof}

\subsection{Target vectors in \texorpdfstring{$P(\lam)$}{P(lam)}}
Let $\lam=a\eps+b\del\in\Lambda_{odd}$. We begin by fixing our basis of $P(\lam)$ in the following way. When $a\geq-1$, so that $P(\lam)=\Del(\lam)$ as seen in Remark \ref{rmk:gen_of_projs}, let the basis of $P(\lam)$ be the PBW basis \eqref{eq:PBW_basis} of $\Del(\lam)$ in its Borel presentation. When $a\leq-3$, let $B_0$ be the basis of $P_0(\lam)$ described in Remark \ref{rmk:gen_of_projs} (visualised in Figure \ref{fig:vecs_to_consider_in_P0}); then we use the following basis, derived from the PBW theorem,
\[
\{y^{\alpha_1}_-h^{\alpha_2}_-x_-^{\alpha_3}\otimes z \ | \ z\in B_0, \ \alpha_1,\alpha_2,\alpha_3\in\{0,1\}\},
\]
where the monomials are written in the order shown. Accordingly we define
\[t_\lam=\begin{cases}
1\otimes w & \text{if } a\geq -1,\\
1\otimes u_0 & \text{if } a\leq -3,
\end{cases}\]
which generates $P(\lam)$.

Let $\mu\in\Lambda_{odd}$. Using Proposition~\ref{prop:dim_counts}, we now find a basis of $\Hom_\g(P(\mu),P(\lam))$. Because $t_\mu$ generates $P(\mu)$, any morphism $P(\mu)\to P(\lam)$ is uniquely determined by the image of $t_\mu$. This motivates the following definition.

\begin{defn}\label{defn:targets}
    Let $\lam,\mu\in\Lambda_{odd}$. Define the following vector space
    \[
        T_\mu(P(\lam)):=\{\phi(t_\mu) \ | \ \phi\in\Hom_\g(P(\mu),P(\lam))\}\subseteq P(\lam)_\mu.
    \]
    
      The \textit{space of target vectors} of $P(\lam)$ is $T(P(\lam)):=\bigoplus_{\mu\in\Lambda_{odd}}T_\mu(P(\lam))$. Elements in $T(P(\lam))$ are called \textit{target vectors}. We call a basis of $T(P(\lam))$ a \textit{target basis}.
    
\end{defn}

For a projective module $P\in\O(\g_0)$ and a weight $\mu\in\Lambda_{odd}$, we similarly define
\[T^0_\mu(P):=\{\phi(u_0) \ | \ \phi\in\Hom_{\g_0}(P_0(\mu),P)\}\subseteq P_\mu,\]
where $u_0$ is the generator of $P_0(\mu)$ fixed in Remark \ref{rmk:gen_of_projs}.

\begin{lem}\label{lem:proj_targets}
    Let $P$ be a projective module in $\O(\g_0)$ and let $\mu\in\Lambda_{odd}$. Then
    \[T^0_\mu(P)=\{v\in P_\mu \ | \ xyx\cdot v=0\}.\]
\end{lem}
\begin{proof}
    Since $xyx\cdot u_0=0$ in $P_0(\mu)$, every $\phi\in\Hom_{\g_0}(P_0(\mu),P)$ satisfies $xyx\cdot\phi(u_0)=\phi(xyx\cdot u_0)=0$. Hence, $T^0_\mu(P)\subseteq\{v\in P_\mu \ | \ xyx\cdot v=0\}$.

    Now suppose $P=P_0(\lam)$ is indecomposable. Because $\phi$ is uniquely determined by
    $\phi(u_0)$, we have
    $\dim T^0_\mu(P_0(\lam))=\dim\Hom_{\g_0}(P_0(\mu),P_0(\lam))$. One can verify directly
    that this agrees with $\dim\{v\in P_0(\lam)_\mu \ | \ xyx\cdot v=0\}$ for all $\lam$, so the
    inclusion above is an equality. The general case follows by writing $P$ as a direct sum
    of indecomposable projectives.
\end{proof}

\begin{cor}\label{cor:ID_targets}
    Let $\lam,\mu\in\Lambda_{odd}$. Then
    \[T_\mu(P(\lam))=\{v\in P(\lam)_\mu \ | \ \px v=0 \text{ and } xyx\cdot v=0\}.\]
\end{cor}
\begin{proof}
    First note that $P(\lam)$ is a projective $\g_0$-module, since $P(\lam)\cong U(\g_{-1})\otimes P_0(\lam)$
    as $\g_0$-modules (where $\g_0$ acts on $U(\g_{-1})$ by $X\cdot u=Xu-uX$) and $U(\g_{-1})$
    is finite dimensional. Now, via the same adjunction between induction and restriction as in the proof of Proposition \ref{prop:projectives},
    \[\Hom_\g(P(\mu),P(\lam))\cong\Hom_{\g_{\geq0}}(P_0(\mu),\Res^\g_{\g_{\geq0}}P(\lam))
    =\Hom_{\g_0}(P_0(\mu),\ker\px(P(\lam))).\] Under this identification a
    morphism $\phi$ with $\phi(t_\mu)=v$ corresponds to the $\g_0$-morphism sending
    $u_0\mapsto v$, so
    \[T_\mu(P(\lam))=\{v\in T^0_\mu(P(\lam)) \ | \ \px v=0\}.\]
    Indeed, if $\psi\in\Hom_{\g_0}(P_0(\mu),P(\lam))$ then $\psi(P_0(\mu))=U(\g_0)\psi(u_0)$
    by Remark \ref{rmk:gen_of_projs}. For $X\in\g_0$ we have $\px X=X\px+[\px,X]$, where
    $[\px,X]=0$ if $X\in\s$ and $[\px,\xi\px]=\px$; in either case $[\px,X]v=0$ whenever
    $\px v=0$. Hence $\px\circ\psi=0$ if and only if $\px\psi(u_0)=0$. Applying Lemma
    \ref{lem:proj_targets} to $P(\lam)$ gives the result.
\end{proof}

\begin{rmk}\label{rmk:logic_for_computing_targets}
Note the following on target vector spaces and target bases.

\begin{enumerate}
    \item Because $t_\mu$ generates $P(\mu)$, evaluation at $t_\mu$ is a vector space isomorphism $\Hom_\g(P(\mu),P(\lam))\to T_\mu(P(\lam))$ and so a basis of $T(P(\lam))$ corresponds to a basis of $\bigoplus_{\mu\in\Lambda_{odd}}\Hom_\g(P(\mu),P(\lam))$. 
    \item Moreover, $\dim T_\mu(P(\lam))=[P(\lam):L(\mu)]$ and so Proposition \ref{prop:dim_counts} allows us to compute $\dim T(P(\lam))$ and also tells us precisely which $\mu\in\Lambda_{odd}$ makes $T_\mu(P(\lam))\neq0$. 
    \item Corollary \ref{cor:ID_targets} tells us the simultaneous equations we need to solve (for each such $\mu$) to find an explicit basis for $T_\mu(P(\lam))$ (and so a basis of $T(P(\lam))$).
\end{enumerate}

\end{rmk}
\subsection{Target basis of \texorpdfstring{$P(\lam)=\Del(\lam)$}{P(lam)=Del(lam)}}
Let $\lam\in\Lambda_{odd}$. We will find a basis for $T(P(\lam))$ when $a\geq -1$. As in Remark \ref{rmk:Odd_projs}, $P(\lam)=\Del(\lam)$ if and only if $a\notin\{-2,-3,-4,\dots\}$. We proceed with the Borel presentation of $\Del(\lam)=P(\lam)$ with the PBW basis \eqref{eq:PBW_basis}.

 \begin{figure}[H]
        \centering
        \adjustbox{scale=.7,center}{%
        \begin{tikzcd}
            &  (b-3)\del  &  (b-2)\del & (b-1)\del & b\del \phantom{textt}  &\\
            (a+2) \eps & & \bullet \arrow[d] \arrow[dl, tail] & \bullet \arrow[d] \arrow[l, two heads] \arrow[dl, tail] &  & \\
            a \eps & \bullet \arrow[d] & \times \arrow[d] & \times \arrow[d] \arrow[dl,  dashed, tail]& \bullet \arrow[d, "y"] \arrow[ul, swap, "x_-"] \arrow[l, two heads, swap, "h_-"] \arrow[dl, tail, "y_-" near start]\\
               (a-2)\eps  & \bullet \arrow[d] & * \arrow[d] & *\arrow[d] & \bullet \arrow[d]\\
            & \vdots  & \vdots & \vdots &\vdots  & 
        \end{tikzcd}
        }
        \caption{Weight space visualisation of $\Del(\lam)$. Weight spaces represented by a $\bullet$ are one dimensional, by a $\times$ are two dimensional and by a $*$ are three dimensional.}
        \label{fig:target_vecs_verm_a>2}
    \end{figure}

Figure \ref{fig:target_vecs_verm_a>2} is a visualisation of the weight spaces of $\Del(\lam)$ and their dimensions. The paths starting at $a\eps+b\del$ in Figure \ref{fig:target_vecs_verm_a>2} almost correspond to the PBW basis of $\Del(\lam)$, except for the fact that $y_-yx_-\otimes w=yy_-x_-\otimes w$. To remedy this, the dotted line represents applying $y_-$ to $h_-\otimes w$, not $yx_-\otimes w$. 

\begin{prop}\label{prop:targets_verm}
    Let $\lam=a\eps+b\del\in\Lambda_{odd}$ with $a\geq -1$.
    \begin{enumerate}[label=\textbf{(\arabic*)},leftmargin=*,itemsep=.8em]
        \item \textbf{The case $a\geq3$.} Here $\dim T(P(\lam))=8$ and a target basis of
        $P(\lam)$ is as follows. The first four target vectors in the basis are
        \begin{enumerate}[label=$(\Alph*)$,leftmargin=*,topsep=.3em,itemsep=.2em]
            \item $1\otimes w$,
            \item $x_-\otimes w$,
            \item $(-(a+2)y_-x_-+yh_-x_-)\otimes w$,
            \item $(-a(a+1)y_-+(a+1)yh_-+y^2x_-)\otimes w$.
        \end{enumerate}
        The remaining four basis elements are obtained by applying $y^{\alpha+1}$ to the
        target vectors $(A)$ to $(D)$, where $\alpha$ is the $\eps$-coefficient of the
        weight of $(A),(B),(C)$ or $(D)$ respectively. Label these target vectors as
        $(X')$ (where $X=A,B,C,D$) to reflect that they lie in the dual weight spaces.

        \item \textbf{The case $a=1$.} Here $\dim T(P(\lam))=7$, and a target basis is
        given by the vectors of case \textbf{(1)} with $a=1$, except that $(D)=(D')$.

        \item \textbf{The case $a=-1$.} Here $\dim T(P(\lam))=5$. The first three target
        vectors in the basis are $(A),(B),(C)$ with $a=-1$, and the remaining two lie in
        the same weight space:
        \begin{enumerate}[label=$(D_{\arabic*})$,leftmargin=*,topsep=.3em,itemsep=.2em]
            \item $y^2x_-\otimes w$ (equal to the vector $(D)$ with $a=-1$),
            \item $(y_-+yh_-)\otimes w$.
        \end{enumerate}
    \end{enumerate}
\end{prop}
\begin{proof}
    Computed by hand and verified with \textit{Mathematica}, using the logic outlined in Remark \ref{rmk:logic_for_computing_targets}.
\end{proof}

\begin{example}\label{ex:worked_target}
    To illustrate the computation in Proposition \ref{prop:targets_verm}, we work through
    the case of the target vector $(C)$. Let $\lam=a\eps+b\del\in\Lambda_{odd}$ with
    $a\geq3$ and take $\mu=\lam-2\del$, so that $\dim T_\mu(P(\lam))=1$ by Proposition
    \ref{prop:dim_counts}. The weight space $\Del(\lam)_\mu$ is two dimensional, spanned
    by the PBW basis vectors $y_-x_-\otimes w$ and $yh_-x_-\otimes w$, so write
    \[v=c_1y_-x_-\otimes w+c_2yh_-x_-\otimes w.\]
    Using the bracket relations of Section \ref{sec:defns} to rewrite $\px v$ and
    $xyx\cdot v$ in the PBW basis, we get
    \[
        \px v=(c_1+(a+2)c_2)\,yx_-\otimes w,\qquad
        xyx\cdot v=(a+2)(c_1+(a+2)c_2)\,h_-x_-\otimes w.
    \]
    So both of the conditions in Corollary \ref{cor:ID_targets} come down to
    $c_1+(a+2)c_2=0$, and $T_\mu(P(\lam))$ is spanned by
    $(-(a+2)y_-x_-+yh_-x_-)\otimes w$ (the target vector $(C)$).
\end{example}

\subsection{Target basis of \texorpdfstring{$P(\lam)\neq\Del(\lam)$}{P(lam) not equal Del(lam)}}\label{sec:targets_non_verma}
We now find a target basis of $P(\lam)$ when $P(\lam)\neq\Del(\lam)$. Some of the computations in this case are much more involved. To help with these computations, we wrote a program in \textit{Mathematica}. The inspiration for such a program came from the Stack Exchange post \cite{lamers_2018}, which outlines how to work with finitely generated algebras and relations in \textit{Mathematica}. The notebook is available as an ancillary file with the arXiv version of this paper.

As discussed in Remark \ref{rmk:g0_sl2}, we already know what $P_0(\lam)$ looks like thanks to our knowledge of $\O(\sl_2)$ (see \cite[Proposition 5.2.12]{Mazorchuk}). We choose a basis of $P_0(\lam)$ to align with vectors shown in Figure \ref{fig:vecs_to_consider_in_P0}. 

\begin{figure}[H]
\centering
    \adjustbox{scale=0.75,center}{%
    \begin{tikzcd}
        & & a-4 & a-2 &a & a+2 & a+4 & \hdots & -a-6 & -a-4 & -a-2\\
        & \hdots \arrow[r, bend left = 10, dashed] & v_2 \arrow[l, bend left = 10 ] \arrow[r, bend left = 10, dashed] & v_1  \arrow[l, bend left = 10 ] \arrow[r, bend left = 10, dashed] & v_0 \arrow[l, bend left = 10 ] \arrow[r, bend left = 10, dashed] & v_{-1} \arrow[l, bend left = 10 ] \arrow[r, bend left = 10, dashed] & v_{-2} \arrow[l, bend left = 10 ] \arrow[r, bend left = 10, dashed] & \hdots \arrow[l, bend left = 10 ] \arrow[r, bend left = 10, dashed] & v'' \arrow[l, bend left = 10 ] \arrow[r, bend left = 10, dashed] & v' \arrow[l, bend left = 10 ] \arrow[r, bend left = 10, dashed, "x"] & v \arrow[l, bend left = 10, "y"] \\
        & \hdots \arrow[r, bend left = 10, dashed] \arrow[ur, dashed] & u_2 \arrow[l, bend left = 10 ] \arrow[r, bend left = 10, dashed] \arrow[ur, dashed] & u_1 \arrow[l, bend left = 10 ] \arrow[r, bend left = 10, dashed] \arrow[ur, dashed] & u_0 \arrow[l, bend left = 10 , "y"] \arrow[ur, dashed, "x",swap] &  & 
    \end{tikzcd}
    }
    \caption{The basis and structure of $P_0(\lam)$ when $a\leq-3$. The elements explicitly depicted are those that, upon tensoring with specific elements in $U(\g_{-1})$, lie in weight spaces of $P(\lam)$ known to contain target vectors (using Proposition \ref{prop:dim_counts}). Note that we omit the component of each weight corresponding to the action of $\xi\px$ for simplicity. }
    \label{fig:vecs_to_consider_in_P0}
\end{figure}

\begin{prop}\label{prop:targets_non_verma}
    Let $\lam=a\eps+b\del\in\Lambda_{odd}$ with $a\leq-3$.
    \begin{enumerate}[label=\textbf{(\arabic*)},leftmargin=*,itemsep=.8em]
        \item \textbf{The case $a\leq-5$.} Here $\dim T(P(\lam))=12$ and a target basis of
        $P(\lam)$ is as follows. The first eight target vectors in the basis are
        \begin{enumerate}[label=$(A_{\arabic*})$,leftmargin=*,topsep=.3em,itemsep=.2em]
            \item $1\otimes v_0$,
            \item $1\otimes u_0$,
        \end{enumerate}
        \begin{enumerate}[label=$(B_{\arabic*})$,leftmargin=*,topsep=.3em,itemsep=.2em]
            \item $x_-\otimes v_0$,
            \item $-y_-\otimes v_{-2}+h_-\otimes v_{-1}-(a+1)x_-\otimes u_0$,
        \end{enumerate}
        \begin{enumerate}[label=$(C_{\arabic*})$,leftmargin=*,topsep=.3em,itemsep=.2em]
            \item $-ay_-x_-\otimes v_0+h_-x_-\otimes v_1$,
            \item $y_-h_-\otimes v_{-1}-ay_-x_-\otimes u_0+h_-x_-\otimes u_1$,
        \end{enumerate}
        \begin{enumerate}[label=$(D_{\arabic*})$,leftmargin=*,topsep=.3em,itemsep=.2em]
            \item $(a^2-a)y_-\otimes v_0-(a-1)h_-\otimes v_1-x_-\otimes v_2$,
            \item $a(a^2-1)y_-\otimes u_0-(a^2-1)h_-\otimes u_1-(a+1)x_-\otimes u_2
                  +(2a-1)y_-\otimes v_0-h_-\otimes v_1$.
        \end{enumerate}
        The remaining four basis elements lie in the dual weight spaces of the above
        target vectors:
        \begin{enumerate}[label=$(\Alph*')$,leftmargin=*,topsep=.3em,itemsep=.2em]
            \item $1\otimes v$,
            \item $-(a+2)(a+3)y_-\otimes v-(a+3)h_-\otimes v'+x_-\otimes v''$,
            \item $h_-x_-\otimes v'+(a+2)y_-x_-\otimes v$,
            \item $x_-\otimes v$.
        \end{enumerate}

        \item \textbf{The case $a=-3$.} Here $\dim T(P(\lam))=11$, and a target basis is
        given by the vectors of case \textbf{(1)} with $a=-3$, except that $(B')$ becomes
        $x_-\otimes v_0$, which is equal to $(B_1)$.
    \end{enumerate}
\end{prop}
\begin{proof}
    Computed by hand and verified with \textit{Mathematica}, using the logic outlined in Remark \ref{rmk:logic_for_computing_targets}.
\end{proof}

\begin{rmk}\label{rmk:target_naming}
    The labels $(A),(B),(C),(D)$ used in Propositions \ref{prop:targets_verm} and \ref{prop:targets_non_verma} are deliberate and indicate the weight of a target vector relative to the
    weight $\lam$ of the projective module $P(\lam)$ containing it. Specifically,
    \begin{center}
    \begin{tabular}{c|c}
        label & weight \\
        \hline
        $(A)$ & $\lam$ \\
        $(B)$ & $\lam+2\eps-\del$ \\
        $(C)$ & $\lam-2\del$ \\
        $(D)$ & $\lam-2\eps-\del$
    \end{tabular}
    \end{center}
    A primed label $(X')$ denotes a target vector in the dual weight space of $(X)$, and
    numerical subscripts $(X_1),(X_2)$ distinguish basis elements when the corresponding
    space of target vectors is two dimensional. 
\end{rmk}

\subsection{Generators and relations}

We begin by stating the final description of $\A=\bigoplus_{\mu,\lam}\Hom_\g(P(\mu),P(\lam))$, the endomorphism algebra of projective modules in $\O_{odd}$, and then explain how it was derived. We start by defining a set of morphisms that we will show generate $\A$.

\begin{defn}\label{def:generators_of_A}
    For each $\mu\in\Lambda_{odd}$ let $\mathrm{id}_{P(\mu)}:P(\mu)\to P(\mu)$ be the identity morphism.
    
    Let $\mu=a\eps+b\del\in\Lambda_{odd}$ with $a\geq3$, and set $K(c)=\frac{2}{c+1}$ for
    $c\leq-3$. Define the following morphisms, where the third column gives the target
    vector in the codomain that induces the morphism (labelled as in Remark
    \ref{rmk:target_naming}), and the last column the scalar used to fix the primed
    generators. Note the subscripts of $\alpha_{-a}$ and $\beta_{-a-4}$ reflect the coefficient of $\eps$ in the codomain of the respective morphisms.

    \begin{center}
    \begin{tabular}{c|c|c|c}
     morphism & domain $\to$ codomain & target vector & coefficient\\
     \hline
     $f_\mu$ & $P(\mu)\to P(\mu+2\eps+\del)$ & $(D)$ &  --\\
     $g_\mu$ & $P(\mu)\to P(\mu-2\eps+\del)$ & $(B)$  & --\\
     $p_\mu$ & $P(\mu)\to P(\mu')$ & $(A')$  & --\\
     $q_{\mu'}$ & $P(\mu')\to P(\mu)$ & $(A')$  & --\\
     $f'_{\mu'}$ & $P(\mu')\to P(\mu'-2\eps+\del)$ & $(B_1),(B_2)$ 
        & $\beta_{-a-4}=K(-a-4)$\\
     $g'_{\mu'}$ & $P(\mu')\to P(\mu'+2\eps+\del)$ & $(D_1),(D_2)$ 
        & $\alpha_{-a}=-K(-a)$\\
    \end{tabular}
    \end{center}

    Here $f'_{\mu'}:=f'_{\mu',2}+\beta_{-a-4}f'_{\mu',1}$ and
    $g'_{\mu'}:=g'_{\mu',2}+\alpha_{-a}g'_{\mu',1}$, where $f'_{\mu',i}$ and $g'_{\mu',i}$
    are the morphisms induced by the target vectors $(B_i)$ and $(D_i)$ respectively. 

    The same morphisms are defined when $a=1$, omitting $g_\mu$ and setting
    $\alpha_{-1}=0$; the latter is forced, as $K(-1)$ is undefined. Note also that the
    codomain of $g'_{\mu'}$ is then $\Del(-\eps+(b+1)\del)$, so $g'_{\mu',1},g'_{\mu',2}$
    are induced by $(D_1),(D_2)$ of Proposition \ref{prop:targets_verm} rather than
    Proposition \ref{prop:targets_non_verma}.

    When $a=-1$ we have $\mu'=\mu$, and we define only $f'_{\mu'}$, with
    $\beta_{-3}=K(-3)$. Note that these omissions are here because the particular morphisms are not needed to generate $\A$ (Lemma \ref{lem:gens_downstairs}). 
\end{defn}

\begin{rmk}
    The ``dual'' notation on the morphisms ($f'_{\mu'},g'_{\mu'}$) is chosen so that the weight of the projective module in their codomains is dual to the weight of the projective in the codomains of $f_{\mu},g_{\mu}$ respectively.
\end{rmk}

We now construct a quiver to describe $\A$. Let $\mu=a\eps+b\del\in\Lambda_{odd}$. Define the \textit{local quiver centred at} $\mu$, for $a\geq3$, $a=1$ and $a=-1$ as the quiver on a sublattice of $\Z\times\Z$ depicted in Figures \ref{fig:local_pic_ageq3}, \ref{fig:local_pic_aeq1} and \ref{fig:local_pic_aeqm1} respectively. Let $V_\mu$ and $E_\mu$ be the vertex set and edge sets of the local quiver centred at $\mu$. Denote by $Q$ the quiver with edges $E_Q=\cup_\mu E_\mu$ and vertices $V_Q=\cup_{\mu}V_\mu=\{(a,b) \ | \ a\eps+b\del\in\Lambda_{odd}\}$. Note that there is a bijection between $E_Q$ and the $ f_{\mu},g_\mu,p_\mu,q_{\mu'},f'_{\mu'},g'_{\mu'}\in\A$ defined in Definition \ref{def:generators_of_A}, which is indicated by the labelling of edges in Figures \ref{fig:local_pic_ageq3}, \ref{fig:local_pic_aeq1} and \ref{fig:local_pic_aeqm1}. Let $\C Q$ be the path algebra of $Q$. We define the algebra homomorphism $\Phi:\C Q\to \A$ by mapping $e_\mu$ to $\mathrm{id}_{P(\mu)}$ (where $e_\mu$ is the stationary path at vertex $\mu$), and the edges of $Q$ to the morphisms from Definition \ref{def:generators_of_A}. Note that all non-labelled edges in Figures \ref{fig:local_pic_ageq3}, \ref{fig:local_pic_aeq1} and \ref{fig:local_pic_aeqm1} are labelled edges of another local quiver.

\begin{figure}[ht]
    \centering
    \begin{subfigure}[b]{0.32\textwidth}
        \centering
        \adjustbox{scale=.55,center}{
            \begin{tikzcd}
                &     & b\phantom{textt}             & b+1               & b+2       \\
                &a+2  &                              &   \bullet \arrow[dr] \arrow[ddddddd, bend right=15]        &           \\
                &a    & \bullet\arrow[ur, "f_{\mu}"] \arrow[dr,swap, "g_\mu"] \arrow[ddddd, bend right=15, swap,  "p_\mu"]  &                   &\bullet \arrow[ddddd, bend right = 15]   \\
                &a-2  &           &   \bullet \arrow[ur] \arrow[ddd, bend right=15]       &           \\
                &  \vdots   &           &                   &           \\
                & \vdots    &           &                   &           \\
                &-a   &           &   \bullet\arrow[dr] \arrow[uuu, bend right = 15]         &           \\
                &-a-2 &   \bullet \arrow[ur, "g'_{\mu'}"] \arrow[uuuuu, bend right = 15, "q_{\mu'}", swap] \arrow[dr, swap, "f'_{\mu'}"] &                   &   \bullet \arrow[uuuuu, bend right = 15] \\
                &-a-4 &           &   \bullet\arrow[ur] \arrow[uuuuuuu, bend right =15]         &           \\
            \end{tikzcd}
            }
        \caption{Local quiver at $a\geq3$}
        \label{fig:local_pic_ageq3}
    \end{subfigure}
    \hfill 
    \begin{subfigure}[b]{0.32\textwidth}
        \centering
        \adjustbox{scale=.55,center}{
            \begin{tikzcd}
                &     & b\phantom{textt}             & b+1               & b+2       \\
                &3  &                              &   \bullet \arrow[dr] \arrow[dddd, bend right=20]        &           \\
                &1    & \bullet\arrow[ur, "f_{\mu}"] \arrow[dd, bend right=15, swap, "p_{\mu}"]  &                   &\bullet \arrow[dd, bend right = 15]   \\
                &-1  &           &   \bullet \arrow[dr]  &           \\
                &-3 &   \bullet \arrow[ur, "g'_{\mu'}"] \arrow[uu, bend right = 15, swap, "q_{\mu'}"] \arrow[dr, swap, "f'_{\mu'}"] &                   &   \bullet \arrow[uu, bend right = 15] \\
                &-5 &           &   \bullet\arrow[ur] \arrow[uuuu, bend right =20]         &           \\
            \end{tikzcd}
            }
        \caption{Local quiver at $a=1$}
        \label{fig:local_pic_aeq1}
    \end{subfigure}
    \hfill 
    \begin{subfigure}[b]{0.32\textwidth}
        \centering
        \adjustbox{scale=.55,center}{
  \begin{tikzcd}
        &     & b\phantom{textt}             & b+1               & b+2       \\
        &1    &  &    \bullet \arrow[dd, bend right = 15, swap]       &   \\
        &-1  &  \bullet \arrow[dr,swap,"f'_{\mu'}"]         &    &        \bullet   \\
        &-3 &   &        \bullet  \arrow[ur,swap] \arrow[uu, bend right = 15, swap]        &    \\
    \end{tikzcd}
    }
        \caption{Local quiver at $a=-1$}
        \label{fig:local_pic_aeqm1}
    \end{subfigure}
    \caption{Local quivers of $\A$}
    \label{fig:local_pics}
\end{figure}

\begin{rmk}
    Note that all weights $\mu\in\Lambda_{odd}$ have at least one morphism defined in Definition \ref{def:generators_of_A} with source $\mu$. Further, every $\mu\in\Lambda_{odd}$ is either the centre of a local quiver (that is, has $\eps$-coefficient at least $-1$) or the dual of exactly one such centre, by construction.
\end{rmk}

\begin{rmk}\label{rmk:local_pic_convention}
    In what follows, unless otherwise stated we will fix a $\mu=a\eps+b\del$ with $a\geq -1$ and write $\mu'$ for its dual (so $\mu=\mu'$ when $a=-1$). Then, we assume the source of any path in $\C Q$ (or corresponding morphisms in $\A$, via $\Phi$) is either $\mu$ or $\mu'$, and the order of composition determines the source and target of each edge in the path. This convention allows us to drop subscripts for morphisms in $\A$ and edges in $Q$ to improve readability. For example, if $a\geq 3$, the path $pgf\in\C Q$ refers to
    \[(a,b)\xrightarrow{ \ f\ }(a+2,b+1)\xrightarrow{\ g\ }(a,b+2)\xrightarrow{\ p\ }(-a-2, b+2)\] or in $\A$,
    \[P(\mu)\xrightarrow{ \ f\ }P(\mu+2\eps+\del)\xrightarrow{\ g\ }P(\mu+2\del)\xrightarrow{\ p\ }P(\mu'+2\del).\] Also, $f^2$ refers to \[P(\mu)\xrightarrow{ \ f\ } P(\mu+2\eps+\del)\xrightarrow{ \ f \ }P(\mu+4\eps+2\del).\]
\end{rmk}

Before the description of $\A$, we first show that $\Phi$ is surjective.

\begin{lem}\label{lem:gens_downstairs}
    For any choice $\alpha_{-a},\beta_{-a-4}\in\C$ in the definition of $f'_\mu,g'_\mu$ (in Definition \ref{def:generators_of_A}), $\Phi$ is surjective.
\end{lem}
\begin{proof}
Let $\mu,\lam\in\Lambda_{odd}$ and $h\in\Hom_\g(P(\mu),P(\lam))$ be a morphism corresponding to one of the target basis vectors outlined in Propositions \ref{prop:targets_verm} or \ref{prop:targets_non_verma}. Note by construction of the local quivers and Remark \ref{rmk:target_naming}, the source weight $\mu$ and target weight $\lam$ correspond to vertices in the local quiver centred at $\mu$ or $\mu'$ (depending on its $\eps$ coefficient). One can check that for any such $h$, there exists a path in $Q$ with source $\mu$ and target $\lam$ whose image under $\Phi$ is equal to $h$ up to scalar. This is a finite check, since the target vectors and relations depend only on $a$ and not on $b$, and can be done by hand in some cases and using our \textit{Mathematica} tool otherwise. 
\end{proof}

\begin{theorem}\label{thm:relations_in_A}
    We have $\C Q/I\cong\A$, where $I$ is the ideal generated by the relations in Table \ref{tab:relations}.
    The `$a$' in the table below is in reference to the fixed weight $\mu=a\eps+b\del$. All relations have source $\mu$ or $\mu'$.
    \begin{table}[ht!]
    \caption{Relations between generators in $\A$.}\label{tab:relations}
    
   \begin{center}
        \begin{tabular}{ c | c | c }
         $a\geq 3$ & $a=1$ & $a=-1$ \\
         \hline 
        $f^2=g^2=f'^2=g'^2=qp=0$ & $f^2=f'^2=qp=0$ & $f'^2=g'f'=0$  \\[.5em]
         $fg+\frac{a-1}{a+3}gf=0$ & - & -\\[.5em] 
         $f'g'+\left(\frac{a-1}{a+3}\right)^2g'f'=0$ & $f'g'-\frac{1}{16}g'f'=0$ & - \\[.5em] 
         $f'p=(a+1)pf$ & $f'p=2pf$ & - \\[.5em] 
         $qf'=(a+3)fq$& $qf'=4fq$ & - \\[.5em]
         $g'p=(a+1)pg$& - & -\\[.5em] 
         $qg'=(a-1)gq$& - &  -
         
        \end{tabular}
    \end{center}
    \end{table}
\end{theorem}

See Appendix \ref{ap:0} for more details and also for why the coefficients in the definition of $f'$ and $g'$ are chosen as such (in short, to have purely quadratic relations). 

\begin{rmk}\label{rmk:grading}
    Note that $\C Q$ is a graded algebra by length of words and $\C Q/I$ inherits this grading because $I$ is a homogeneous ideal. Further, by Theorem \ref{thm:relations_in_A}, $\A$ is graded in the same way. 
\end{rmk}

\begin{lem}\label{lem:linear_basis_of_A}
Fix a $\mu=a\eps+b\del\in\Lambda_{odd}$ with $a\geq -1$ and recall that $\mu'=(-a-2)\eps+b\del$ is the dual of $\mu$. When $a\geq3$, the following is a graded basis of morphisms with source $\mu$ or $\mu'$ in $\C Q /I$.  
\begin{table}[H]
    \centering
    \begin{tabular}{c|c|c|c|c|c|c}
        deg=1 & $f$ & $g$ & $p$ & $f'$ & $g'$  & q \\
        \hline 
        target & $\mu+2\eps+\del$ & $\mu-2\eps+\del$ & $\mu'$ & $\mu'-2\eps+\del$ & $\mu'+2\eps+\del$ & $\mu$\\
        \hline
        source & \multicolumn{3}{|c|}{$\mu$}  & \multicolumn{3}{|c}{$\mu'$}\\
    \end{tabular}
\end{table}
\begin{table}[H]
    \centering
    \begin{tabular}{c|c|c|c|c|c|c|c}
        deg=2 & $gf$ & $g'p$ & $f'p$ & $pq$ & $g'f'$  & $qf'$ & $qg'$ \\
        \hline 
        target & $\mu+2\del$ & $\mu'+2\eps+\del$ & $\mu'-2\eps+\del$ & $\mu'$ & $\mu'+2\del$ & $\mu+2\eps+\del$& $\mu-2\eps+\del$   \\
        \hline
        source & \multicolumn{3}{|c|}{$\mu$}  & \multicolumn{4}{|c}{$\mu'$}\\
    \end{tabular}
\end{table}
\begin{table}[H]
    \centering
    \begin{tabular}{c|c|c|c|c}
        deg=3 & $g'f'p$ & $gfq$ & $pqf'$ & $g'pq$   \\
        \hline 
        target & $\mu'+2\del$ & $\mu+2\del$ & $\mu'-2\eps+\del$ & $\mu'+2\eps+\del$ \\
        \hline
        source & \multicolumn{1}{|c|}{$\mu$}  & \multicolumn{3}{|c}{$\mu'$}\\
    \end{tabular}
\end{table}
\begin{table}[H]
    \centering
    \begin{tabular}{c|c}
        deg=4 & $g'pqf'$   \\
        \hline 
        target & $\mu'+2\del$ \\
        \hline
        source & $\mu'$
    \end{tabular}
\end{table}

When $a=1$, the basis is the same as the $a\geq3$ case, just without $g$ and $gq$. 

When $a=-1$, the basis is $f', qf', pqf', g'pqf'$ (recall $g'f'=0$ in this case).

\end{lem}
\begin{proof}
    This is quickly verified using that for $a\geq3$, $fg{\sim}gf$, $f'g'{\sim}g'f'$, $fq{\sim}qf'$, $pf{\sim}f'p$, $gq{\sim}qg'$ and $pg{\sim}g'p$, 
    where $\sim$ here means ``equal up to scalar''. The same holds for $a=\pm1$, with appropriate exceptions.
\end{proof}

\begin{proof}[Proof of Theorem \ref{thm:relations_in_A}]
With the help of Mathematica, one can verify that $I\subseteq\ker\Phi$, and so Lemma \ref{lem:gens_downstairs} implies a surjection $\C Q/I\to \A$. Lemma \ref{lem:linear_basis_of_A} and Proposition \ref{prop:dim_counts} allow us to verify by a dimension count that 
\[e_\lam(\C Q/I)e_\mu\cong\Hom_{\g}(P(\mu),P(\lam))\] for all $\mu,\lam\in\Lambda_{odd}$, where $e_\lam$ is the primitive idempotent corresponding to $P(\lam)$ via $\C Q/I\to\A$. 

The relations generating $I$ are all computed directly, either by hand or with Mathematica.

\end{proof}

\begin{cor}
    The endomorphism algebra of $\O_{odd}$ is quadratic.
\end{cor}

\section{Koszulity of \texorpdfstring{$\O_{odd}$}{Oodd koszul}}\label{sec:koszul}

In this section, we will prove that $\O_{odd}$ is Koszul, and compute all higher order extension groups $\ext_{\O_{odd}}^n(L(\mu), L(\lam))$. For details on Koszulity and extensions in graded categories see \cite[\S 2]{beilinson_ginzburg_soergel_1996} or \cite[\S 2]{MAZORCHUK_2010}, and for the formal setting to discuss Koszulity of infinite dimensional algebras see \cite[Definition 1]{Mazorchuk_Ovsienko_Stroppel}.

Throughout this section we use Theorem \ref{thm:relations_in_A} to identify $\A$ with $\C Q/I$, and write $e_\lam$ for the idempotent of $\C Q/I$ given by the stationary path at $\lam$.

By Remark \ref{rmk:grading} and Lemma \ref{lem:linear_basis_of_A}, $\A=\bigoplus_{n=0}^4\A_n$ is graded by length of words in the generators $f,g,p,q,f'$ and $g'$, with $\A_0$ having basis consisting of all idempotents $e_\lam$. 
Let $\A$-Mod and $\A$-mod be the category of $\A$-modules and graded $\A$-modules respectively. Let $\ext$ and $\mathrm{ext}$ be the extensions in these abelian categories (respectively).

\begin{defn}
    Define $\mathcal{P}(\lam) := \A e_\lam$ to be the projective $\A$-module of the idempotent $e_\lam$.
    Let $\A_0=\A/\A_{>0}$ and let $\C e_\lambda$ be the simple $\A$-module equal to the image of $\P(\lam)$ under the natural map $\A\to\A_0$. Finally, for $M\in \A$-mod and $n\in\Z$, we write $M\left<n\right>\in \A$-mod for the module satisfying $M\left<n\right>_i=M_{i-n}$ for all $i\in\Z$.
\end{defn}

\begin{defn}
    Let $M\in\A$-mod be a graded $\A$-module. We can write $M=\bigoplus_{i,\lam}(e_\lam M)_i$, where the sum is over $i\in\Z$ and $\lam\in\Lambda_{odd}$. 
    \begin{enumerate}
    \item We say an element $m\in M$ is \textit{homogeneous} if $m\in (e_\lam M)_i$ for some $i,\lam$.
    
    \item A \textit{homogeneous generating set} of $M$ is a generating set of $M$ consisting only of homogeneous elements.

    \item  A \textit{homogeneous basis} of $M$ is a basis of $M$ consisting of homogeneous elements.
    \end{enumerate}
    
\end{defn}

Fix $\mu\in\Lambda_{odd}$, and take a minimal graded projective resolution (to be specified
shortly)
\begin{equation}\label{eq:full_reso_mu}
    \begin{tikzcd}
        \dots \arrow[r] & \P_2^\mu \arrow[r] & \P_1^\mu \arrow[r]  & \P_0^\mu \arrow[r,  two heads] & \C e_\mu.
    \end{tikzcd}
\end{equation}

We denote the boundary maps by $\phi_n^\mu : \P_n^\mu\to \P_{n-1}^\mu$ for $n\geq 1$ and $\phi_0^\mu: \P_0^\mu\to\C e_\mu$, where $\img\phi_{n+1}^\mu=\ker\phi_n^\mu$ for all $n\geq0$ (as \eqref{eq:full_reso_mu} is a resolution).  By definition, $\A$ is Koszul if and only if there exists such a resolution where the $\A$-module $\P^\mu_n$ is generated in degree $n$, for all $\mu\in\Lambda_{odd}$ (\cite[\S 5.4]{Mazorchuk_Ovsienko_Stroppel}).

Let $\P_0^\mu=\P(\mu)$ with boundary map $\phi_0^\mu:\P_0^\mu\to \C e_\mu$ sending $e_\mu\mapsto e_\mu$. We construct the resolution \eqref{eq:full_reso_mu} inductively. Let $G^\mu_{n-1}$ be a minimal homogeneous generating set of the $\A$-module $\ker\phi^\mu_{n-1}\subseteq \P_{n-1}^\mu$. By definition, for all $h\in G_{n-1}^\mu$ we can define $\lam_h$ which satisfies $h\in e_{\lam_h}\P_{n-1}^\mu$.

Moreover, because $h$ is homogeneous, we can define  \[\P^\mu_n:=\bigoplus_{h\in G_{n-1}^\mu}\P(\lambda_h)\left<\deg(h)\right>,\]
 and let $\phi_n^\mu$ be the morphism which sends $e_{\lam_h}\mapsto h$ for all $h\in G_{n-1}^\mu$. Note that this resolution is a minimal graded projective resolution, but $\P_n^\mu$ need not be generated in degree $n$. 

To prove that $\A$ is Koszul, we will define a subcomplex $(\P_n^\mu)^n\subseteq \P_n^\mu$ which is the ``linear part'' of \eqref{eq:full_reso_mu} (in the sense that $(\P_n^\mu)^n$ is generated in degree $n$) and then prove that the complex $(\P_n^\mu)^n$ (which is well defined by Lemma \ref{lem:on_reso}) is exact; it follows by minimality of \eqref{eq:full_reso_mu} that $\A$ is Koszul. 

\begin{rmk}
    Note that this proof strategy, and everything up to and including Lemma \ref{lem:on_reso}, works for any positively graded algebra $\A$.
\end{rmk}

To be precise, let $(G_{n-1}^\mu)^n:=\{h\in G_{n-1}^\mu \ | \ \deg(h)=n\}$ and \[(\P_n^\mu)^n:=\bigoplus_{h\in (G_{n-1}^\mu)^n}\P(\lambda_h)\left<\deg(h)\right>\subseteq\P_{n}^\mu.\] Further, define $(\phi_n^\mu)^n$ as the restriction of $\phi_n^\mu$ onto $(\P_n^\mu)^n$.

\begin{lem}\label{lem:on_reso}
    For all $n\geq0$ and $\mu,\lam\in\Lambda_{odd}$,
    \begin{enumerate}
        \item Idempotents $e_\lam\in\P_n^\mu$ have $\deg(e_\lam)\geq n$ and all homogeneous $h\in\ker\phi_n^\mu$ satisfy  $\deg(h)\geq n+1$,
        \item The image $\img(\phi_n^\mu)^n$ is contained in $(\P_{n-1}^\mu)^{n-1}$,
        \item We have $\img(\phi^\mu_{n+1})^{n+1}\subseteq\ker(\phi_n^\mu)^n$.
    \end{enumerate}
\end{lem}
\begin{proof}    
    We verify part (1) by induction. For $n=0$, the claim is clearly true. Proceeding by induction, because $\ker\phi_{n-1}^\mu$ is generated in degree greater than or equal to $n$, the idempotents in $\P_n^\mu$ have degree greater than or equal to $n$ by definition. This means that any $h\in\P_n^\mu$ has $\deg(h)\geq n$. But if $h\in\ker\phi_{n}^\mu$ and $\deg(h)=n$, then $h$ must be a linear combination of idempotents and minimality of the generating set $G_{n-1}^\mu$ would be contradicted.
    
    For part (2), it is sufficient to show that for all $e_\lam\in(\P_n^\mu)^n$ we have $\phi_n^\mu(e_\lam)\in(\P_{n-1}^\mu)^{n-1}$. Let $h_\lam:=\phi_n^\mu(e_\lam)$, then $\deg(h_\lam)=n$ because $\deg(e_\lam)=n$. From part (1), we know all idempotents in $\P_{n-1}^\mu$ have degree greater than or equal to $n-1$. Moreover, $h_\lam\in\ker\phi_{n-1}^\mu$ cannot be a linear combination of idempotents. Hence, it must be that $h_\lam\in(\P_{n-1}^\mu)^{n-1}$ by degree consideration. 

    Part (3) follows because $\img\phi_{n+1}^\mu\subseteq\ker\phi_n^\mu$ and part (2).   
\end{proof}

Because of Lemma \ref{lem:on_reso}, we have the following well defined complex, where the boundary maps are $(\phi_n^\mu)^n:(\P_n^\mu)^n\to(\P_{n-1}^\mu)^{n-1}$

\begin{equation}\label{eq:tighter_reso_mu}
    \begin{tikzcd}
            \dots \arrow[r] & (\P_2^\mu)^2 \arrow[r] & (\P_1^\mu)^1 \arrow[r]  & \P_0^\mu \arrow[r,  two heads] & \C e_\mu.
    \end{tikzcd}
\end{equation}

\begin{defn}
    For weights $\mu,\lam\in \Lambda_{odd}$, let $N^i_j(\mu,\lam)$ be the multiplicity of $\P(\lam)\left<j\right>$ in $(\P_i^\mu)^i$ of complex \eqref{eq:tighter_reso_mu}. 
\end{defn}

\begin{lem}\label{lem:dim_comps}
    Let $H^n(\C e_\mu,\C e_\lam\left<n\right>)$ and $h^n(\C e_\mu,\C e_\lam\left<n\right>)$ be the $n$th cohomologies after applying $\Hom(-,\C e_\lam\left<n\right>)$ to complexes \eqref{eq:full_reso_mu} and \eqref{eq:tighter_reso_mu} (respectively). Then \[\dim [h^n(\C e_\mu,\C e_\lam\left<n\right>)]= N^n_n(\mu,\lam)=\dim [H^n(\C e_\mu,\C e_\lam\left<n\right>)].\]
\end{lem}
\begin{proof}
    Easily verified using the definition of the complexes \eqref{eq:full_reso_mu} and \eqref{eq:tighter_reso_mu}.
\end{proof}

We now compute the complex \eqref{eq:tighter_reso_mu} iteratively (which immediately gives us $N^n_n(\mu,\lam)$), and then prove that \eqref{eq:tighter_reso_mu} is exact for all $\mu\in\Lambda_{odd}$.

\begin{rmk}\label{rmk:nature_of_Gn}
    If $h\in(G_{n}^\mu)^{n+1}$, then $\deg(h)=n+1$ and $h\in(\P_{n}^\mu)^n$. So by degree consideration using part (1) of Lemma \ref{lem:on_reso}, $h$ must be one degree above the idempotents and so a linear combination of elements in $\A_1$. By minimality of $G_n^\mu$, $(G_n^\mu)^{n+1}$ must be a homogeneous basis of the degree $n+1$ elements in $\ker\phi_n^\mu$. So, to compute the complex \eqref{eq:tighter_reso_mu} iteratively, we find a homogeneous basis of degree $n+1$ elements in $\ker\phi_n^\mu$.
\end{rmk}

Let $\mu=a\eps+b\del$ with $a\geq5$ and consider the complex \eqref{eq:tighter_reso_mu}. To help keep track of our computation, we use pictures like the one depicted in Figure \ref{fig:extensions_comp}.

\begin{figure}[H]
  \centering
  \adjustbox{scale=.60,center}{%
     \begin{tikzcd}
    &    & b & b+1           & b+2               \\
    &a+4  &          &    &   \bullet        \\
    &a+2  &          &    \circ  \arrow[ur] \arrow[dr] \arrow[dddddddd, bend left = 15]                &           \\
    &a    & \bullet \arrow[dddddd,bend left=15, swap, dotted] \arrow[ur, dotted] \arrow[dr, dotted] &  &          \bullet       \\
    &a-2  &   &   \circ  \arrow[dr] \arrow[ur] \arrow[dddd, bend left = 15]    &               \\
    &a-4  &   &         &      \bullet         \\
    &  \vdots   &  &         &                         \\
    & \vdots    &  &         &                      \\
    &-a   &  &   \bullet       &            \\
    &-a-2 &  \circ \arrow[uuuuuu, bend left = 15] \arrow[ur] \arrow[dr] &   &                  \\
    &-a-4 &  &    \bullet    &            \\
\end{tikzcd}
}
  \caption{Diagram relating to the $n=1,2$ computation of complex \eqref{eq:tighter_reso_mu} with $\mu=a\eps+b\del$ and $a\geq 5$.}
  \label{fig:extensions_comp}
\end{figure}

In Figure \ref{fig:extensions_comp}, open circles correspond to the idempotents in $(\P^\mu_1)^1$, and the filled-in circles correspond to idempotents in $(\P^\mu_2)^2$. The dotted arrows are the elements of $(G_0^\mu)^1$ and appropriate linear combinations of the full arrows are the elements of $(G_1^\mu)^2$. To iterate this picture, discard dotted lines and open circles, turn filled-in circles to open circles and full lines to dotted lines and draw full lines for the homogeneous degree $3$ basis elements of $\ker\phi_2^\mu$ (which will require explicit computations, as we will see in the proof of Theorem \ref{thrm:extn}); this basis is $(G_2^\mu)^3$. The filled in circles correspond to the idempotents of $(\P_3^\mu)^3$.

For $n\geq1$, write $\Lambda_n^\mu:=\{\lam_h \ | \ h\in(G_{n-1}^\mu)^n\}$ for the set of
weights $\lam$ such that $e_\lam$ is an idempotent of $(\P_n^\mu)^n$ (also define
$\Lambda_0^\mu:=\{\mu\}$); these are the open circles in Figure \ref{fig:extensions_comp}.

\begin{lem}\label{lem:which_combinations}
    Let $n\geq0$ and $\mu,\lam\in\Lambda_{odd}$. Then
    \[
        e_\lam\left((\P_n^\mu)^n\right)_{n+1}
        =\mathrm{span}\{ \, x e_{\nu} \ | \ \nu\in\Lambda_n^\mu, \ x \text{ an edge of }
        Q \text{ from } \nu \text{ to } \lam \, \},
    \]
    and
    \[
        \left(\ker\phi_n^\mu\right)_{n+1}
        =\bigoplus_{\lam\in\Lambda_{odd}} e_\lam\left(\ker\phi_n^\mu\right)_{n+1}.
    \]
    Moreover, for each $\lam$ there are at most three edges of $Q$ with target $\lam$, with
    sources $\lam-2\eps-\del$, $\lam+2\eps-\del$ and $\lam'$. The edge from
    $\lam-2\eps-\del$ is $f$ or $g'$, the edge from $\lam+2\eps-\del$ is $g$ or $f'$, and
    the edge from $\lam'$ is $p$ or $q$, according to whether the source lies on the
    $\mu$ or the $\mu'$ side of its local quiver.
\end{lem}
\begin{proof}
    By Remark \ref{rmk:nature_of_Gn}, every homogeneous degree $n+1$ element of $(\P_n^\mu)^n$ is a
    linear combination of elements $xe_\nu$ with $\nu\in\Lambda_n^\mu$ and $x\in\A_1$;
    such an element lies in $e_\lam(\P_n^\mu)^n$ precisely when each $x$ is (a combination
    of) edges from $\nu$ to $\lam$. The direct sum decomposition follows since $\phi_n^\mu$
    is a morphism of graded $\A$-modules, so preserves the decomposition
    $\P_{n-1}^\mu=\bigoplus_{\lam}e_{\lam}\P_{n-1}^\mu$. The last statement is read off
    Definition \ref{def:generators_of_A}.
\end{proof}

Thus, to compute $(G_{n}^\mu)^{n+1}$ we may work one weight $\lam$ at a time, and for
each $\lam$ we need only determine which linear combinations of the at most three edges into
$\lam$ lie in $\ker\phi_{n}^\mu$. 

\begin{defn}\label{defn:M_n}
     For $n\geq 1$ let \[M_n:=\{2n,2n-4, \dots, -2n+4,-2n\},\] and set $M_0:=\{0\}$ and $M_{-1}:=\varnothing$. Denote by \[M_n^{>k}=\{m\in M_n \ | \ m > k\} \ \ \text{ and } \ \ M_n^{<k}=\{m\in M_n \ | \ m < k\}.\]
\end{defn}

\begin{lem}\label{lem:coefficients}
    Fix $\mu=a\eps+b\del\in\Lambda_{odd}$. In the tables below, we specify the elements of $(G_{n-1}^\mu)^n$. In each case
    the element lies in $e_\lam\P_{n-1}^\mu$ for the weight $\lam$ listed, with $m$ ranging
    over the permissible values defined in the last column (using Definition \ref{defn:M_n}). The coefficients are specified implicitly from recurrence relations, with explicit base and extremal cases. Throughout, a coefficient whose defining denominator vanishes corresponds to a morphism which is not present in the linear combination, and is set to zero.


    \bigskip
    \noindent\textbf{(1) The case $a\leq-1$.}
    \begin{center}
    \begin{tabular}{c|c|c}
        element of $(G_{n-1}^\mu)^n$ & $\lam$ & $m$ \\
        \hline
        $\sig^{n,m_1}f'+\tau^{n,m_1}g'$ & $\mu+m_1\eps+n\del$ & $M_n^{<1-a}$\\[.4em]
        $q+\beta^{n,m_2}f+\gamma^{n,m_2}g$ & $\mu'+m_2\eps+(n-1)\del$
            & $M_{n-1}^{>2(n+a)}$\\
    \end{tabular}
    \end{center}

    \textbf{Base case:} When $n=1$ and $a<-1$ let $\sig^{1,2}=0,\tau^{1,2}=1$, $\sig^{1,-2}=1,\tau^{1,-2}=0$ and $\beta^{1,0}=\gamma^{1,0}=0$. When $a=-1$ and $n=1$, we only need to set $\sig^{1,-2}=1,\tau^{1,-2}=0$ (because of the restrictions on $m$).

    \textbf{Extremal cases:} For $n\geq 2$ set
    \[\sig^{n,-2n}=1, \ \tau^{n,-2n}=0\]
    and for $m_1^{\max}:=\max(M_{n}^{<1-a})$ we have
    \[
    \sig^{n,m_1^{\max}}=
    \begin{cases}   
        0, \ &\text{ if } 2n<1-a,\\
        0, \ &\text{ if } 2n\geq1-a \text{ and } a+m_1^{\max}=-1,\\
        \dfrac{1}{\tau^{n-1,m_1^{\max}+2}}, \ &\text{ if } 2n\geq1-a \text{ and } a+m_1^{\max}=-3
    \end{cases}
    \]

    \[
    \tau^{n,m_1^{\max}}=
    \begin{cases}   
        1, \ &\text{ if } 2n<1-a,\\
        1, \ &\text{ if } 2n\geq1-a \text{ and } a+m_1^{\max}=-1,\\
        -\dfrac{1}{16\,\sig^{n-1,m_1^{\max}-2}}, \ &\text{ if } 2n\geq1-a \text{ and } a+m_1^{\max}=-3.
    \end{cases}
    \]
    \textbf{Recursive relations:} All other coefficients satisfy the following recurrence relations
     \begin{align*}
        \sig^{n,m_1}&=\frac{1}{\tau^{n-1,m_1+2}},
        &\tau^{n,m_1}&=\frac{1}{\sig^{n-1,m_1-2}}\left(\frac{a+m_1+3}{a+m_1-1}\right)^{2},\\[.8em]
        \beta^{n,m_2}&=(a-m_2+1)\,\sig^{n-1,-m_2},
        &\gamma^{n,m_2}&=(a-m_2+1)\,\tau^{n-1,-m_2}.
    \end{align*}

    \bigskip
    \noindent\textbf{(2) The case $a\geq1$.}
    \begin{center}
    \begin{tabular}{c|c|c}
        element of $(G_{n-1}^\mu)^n$ & $\lam$ & $m$ \\
        \hline
        $\hat\sig^{n,m_1}f+\hat\tau^{n,m_1}g$ & $\mu+m_1\eps+n\del$ & $M_n^{>2(n-a-1)}$\\[.4em]
        $p+\hat\beta^{n,m_2}f'+\hat\gamma^{n,m_2}g'$
            & $\mu'+m_2\eps+(n-1)\del$ & $M_{n-1}^{<-2(n-a-2)}$\\[.4em]
        $q+v^{n,m_3}f+w^{n,m_3}g$ & $\mu+m_3\eps+(n-2)\del$ & $M_{n-2}^{>2(n-a-3)}$\\
    \end{tabular}
    \end{center}

    \textbf{Base case:} For $n=1$ and $a>1$, set $\hat{\sig}^{1,2}=1,\hat{\tau}^{1,2}=0$, $\hat{\sig}^{1,-2}=0,\hat{\tau}^{1,-2}=1$ and $\hat{\beta}^{1,0}=\hat{\gamma}^{1,0}=0$. For $n=1$ and $a=1$, we only need to set $\hat{\sig}^{1,2}=1,\hat{\tau}^{1,2}=0$ and $\hat{\beta}^{1,0}=\hat{\gamma}^{1,0}=0$.

    For $n=2$, set $v^{2,0}=w^{2,0}=0$, and let the remaining coefficients be given by the relations below.
    
    \textbf{Extremal cases:}
    For $n\geq 2$ set
    \[\hat{\sig}^{n,2n}=1,\quad\hat{\tau}^{n,2n}=0.\]
    When $2n<a+1$ set
    \[\hat{\sig}^{n,-2n}=0,\quad\hat{\tau}^{n,-2n}=1,\]
    and setting $m_2^{\max}:=\max(M_{n-1}^{<-2(n-a-2)})$, when $2n\leq a+1$ set
    \[\hat{\beta}^{n,m_2^{\max}}=0.\]

    \textbf{Recursive relations:}
        All other coefficients satisfy the following recurrence relations
   
    \begin{align*}
        \hat{\sig}^{n,m_1}&=\frac{1}{\hat{\tau}^{n-1,m_1-2}},
        &\hat{\tau}^{n,m_1}&= \frac{1}{\hat{\sig}^{n-1,m_1+2}}\left(\frac{a+m_1-1}{a+m_1+3}\right)
           ,\\[.8em]
        \hat{\beta}^{n,m_2}&=-\frac{\hat{\sig}^{n-1,-m_2}}{a-m_2-1},
        &\hat{\gamma}^{n,m_2}&=-\frac{\hat{\tau}^{n-1,-m_2}}{a-m_2+3}.
    \end{align*}
    \noindent Additionally to these relations, when $n\geq 3$ we have
    \[v^{n,m_3}=-(a+m_3+1)\,\hat{\beta}^{n-1,-m_3},
        w^{n,m_3}=-(a+m_3+1)\,\hat{\gamma}^{n-1,-m_3},\]
        satisfying the following when $m_3\neq1-a$
    \[v^{n,m_3}=\left(\frac{a+m_3+1}{a+m_3-1}\right)\hat{\sig}^{n-2,m_3},\qquad
      w^{n,m_3}=\left(\frac{a+m_3+1}{a+m_3+3}\right)\hat{\tau}^{n-2,m_3}.\]

\end{lem}
\begin{proof}
    The proof follows from careful bookkeeping and an inductive argument. As in the statement, the problem is split first into the $a\geq 1$ and $a\leq -1$ cases. For each case individually, start by proving the result for the ``generic'' case, when the restrictions on $M_n$ are vacuous ($2n<1+a$, $2n<3+a$, $2n<5+a$ for the three rows respectively when $a\geq 1$ and $2n<1-a$ for $a\leq -1$). Then, make a slight amendment to that proof to correctly formulate what happens when the restrictions are not vacuous. The relations between the coefficients are derived using the relations between the generators of $\A$, and the inductive description of $(G_{n-2}^\mu)^{n-1}$, noting that the extremal cases in the $a\leq -1$ case come from the relations $f'g'-1/16g'f'=0$ (local quiver $a=1$) and $g'f'=0$ (local quiver $a=-1$), where the latter is why the constraint $m_1<1-a$ is independent of $n$. Lemma \ref{lem:which_combinations} tells us which linear combinations we need to check for a given $\lam$. The full details for the $a\leq-1$ case are in Appendix \ref{ap:1}, the $a\geq 1$ case follows from a similar argument and the details are not included for brevity.
\end{proof}

\begin{theorem}\label{thrm:extn}
    Let $n\geq 1$, $\mu,\lam\in\Lambda_{odd}$ with $\mu=a\eps+b\del$. Then,


    \begin{itemize}
            \item If $a\leq-1$: 
            \[
                N^n_n(\mu,\lam)
                = \begin{cases}
                    1, \ \mathrm{if } \ \lam = \mu + m\eps +n\del, \text{ for } m\in M_n^{<1-a}\\
                    1, \ \mathrm{if } \ \lam = \mu' + m\eps + (n-1)\del, \text{ for } m\in M_{n-1} ^{> 2(n+a)}  \\
                    0,\ \mathrm{otherwise}
                \end{cases}
            \]
                  \item  If $a\geq 1$:
            \[
            N^n_n(\mu,\lam)= 
            \begin{cases} 
                1, \ \mathrm{if } \ \lam = \mu + m\eps+n\delta, \text{ for } m\in M_n^{>2(n-a-1)}\\
                1, \ \mathrm{if } \ \lam = \mu'+ m\eps+(n-1)\delta, \text{ for } m\in M_{n-1}^{< -2(n-a-2)}\\
                1, \ \mathrm{if } \ \lam = \mu + m\eps+(n-2)\delta, \text{ for } m\in M_{n-2}^{> 2(n-a-3)}\\
                0, \ \mathrm{otherwise}
            \end{cases} 
            \]

    \end{itemize}
\end{theorem}
\begin{proof}
    An immediate consequence of Lemma \ref{lem:coefficients}, where $N^n_n(\mu,\lam)$ is the number of $h\in(G_{n-1}^\mu)^n$ with $h\in e_\lam\P_{n-1}^\mu$.
\end{proof}

\begin{example}
    Take $\mu=5\eps+b\del$, so $\mu'=-7\eps+b\del$. Theorem \ref{thrm:extn} gives
    $N^n_n(\mu,\lam)=1$ for exactly the $\lam$ listed below, and $0$
    otherwise; only the $\eps$-coefficient of $\lam$ is recorded, the $\del$-coefficient
    being given by the column.
    \begin{center}
    \begin{tabular}{c|l|l|l}
        $n$ & $\del$-coefficient $b+n$ & $b+n-1$ & $b+n-2$ \\
        \hline
        $1$ & $3\eps,\ 7\eps$          & $-7\eps$                     & --\\
        $2$ & $\eps,\ 5\eps,\ 9\eps$   & $-9\eps,\ -5\eps$            & $5\eps$\\
        $3$ & $3\eps,\ 7\eps,\ 11\eps$ & $-11\eps,\ -7\eps,\ -3\eps$  & $3\eps,\ 7\eps$
    \end{tabular}
    \end{center}
\end{example}

\begin{theorem}\label{thm:koszul}
    The category $\O_{odd}$ is Koszul. 
\end{theorem}
\begin{proof}   
    We prove that the complex \eqref{eq:tighter_reso_mu} is exact. Then by minimality of \eqref{eq:full_reso_mu}, the resolutions \eqref{eq:tighter_reso_mu} and \eqref{eq:full_reso_mu} are equal and $\A$ is Koszul. That is, we prove that $(G_n^\mu)^{n+1}$ (as described in Lemma \ref{lem:coefficients}) generates $\ker\phi_n^\mu$ for all $n\in\N$ and $\mu\in\Lambda_{odd}$.

    The case for $n=1$ is trivial and $n=2$ is satisfied because $\A$ is quadratic. 

     Let $\mu=a\eps+b\del\in\Lambda_{odd}$. Inducting on $n>2$, we enumerate a graded basis of $\ker\phi_{n-1}^\mu$ (of degree greater than $n$) in Tables \ref{tab:aleq-1} and \ref{tab:ageq1} (depending on the nature of $a$). We use the inductive assumption that $\ker\phi_{n-2}^\mu$ is generated by $(G_{n-2}^\mu)^{n-1}$ and the graded basis of $\A$ defined in Lemma \ref{lem:linear_basis_of_A}. The permissible values of $m$ lie in subsets of $M_n$ analogous to those in Lemma \ref{lem:coefficients}. As noted below, the specific subsets do not affect the argument: whenever a value of $m$ falls outside, the corresponding morphism is absent from the combination and its coefficient is zero, while the relations among the remaining coefficients are as given in Lemma \ref{lem:coefficients}. 
     

    For the min and max values of $m$ in these appropriate subsets of $M_{n}$, certain elements in the combination column of Tables \ref{tab:aleq-1} and \ref{tab:ageq1} may not be present. For example, when $m=2(n-1)$ and $2n<1-a$, $g'p$ does not exist in the combination involving $pq, g'p, f'p$. We interpret this as setting its coefficient to zero in the conditions and relations outlined in this proof. Further, one can verify that denominators in the below relations vanish if and only if the corresponding morphism is not included in the linear combination.

    For each degree and target in Tables \ref{tab:aleq-1} and \ref{tab:ageq1}, we provide the necessary and sufficient conditions on the coefficients for the linear combination to lie in $\ker\phi_{n-1}^\mu$, and then explicitly provide how it is generated by $(G_{n-1}^\mu)^n$. These conditions are computed using the relations found in Lemma \ref{lem:coefficients}.

 \begin{table}[H]
    \centering
    \caption{Combinations for $a\leq-1$}\label{tab:aleq-1}
        \begin{tabular}{c|c|c}
            path length & target & combination   \\
            \hline
            $2$ & $\mu+m\eps+(n+1)\del, \ m\in M_{n-1}$ & $g'f'$  \\
            \hline
            $2$ & $\mu'+m\eps+n\del, \ m\in M_{n-2}$ & $gf,qf',qg'$  \\
            \hline
            $2$ & $\mu+m\eps+(n-1)\del, \ m\in M_{n-1}$ & $pq,g'p,f'p$  \\
            \hline 
            $3$ & $\mu'+m\eps+(n+1)\del, \ m\in M_{n-1}$ & $gfq$  \\
            \hline 
            $3$ & $\mu+m\eps+n\del, \ m\in M_{n-2}$ & $g'f'p,pqf',g'pq$  \\
            \hline 
            $4$ & $\mu+m\eps+(n+1)\del, \ m\in M_{n-1}$ & $g'pqf'$  \\
        \end{tabular}
    \end{table}

    The combinations in Tables \ref{tab:aleq-1} and \ref{tab:ageq1} involving a single element are easily verified to be in $\ker\phi_{n-1}^\mu$ and generated by $(G_{n-1}^\mu)^n$. For example, $g'f'\in\ker\phi_{n-1}^\mu$ because $f'^2=g'^2=0$, and moreover $g'f'=g'(\sig f' + \tau g')/\sig$ (for appropriate $\sig,\tau$).

    Let $A,B,C\in\C$. We list the coefficient conditions and generating identities for the remaining combinations in Table \ref{tab:aleq-1}. Note that the relations in Lemma \ref{lem:coefficients} are used to arrive at the following results.

\begin{enumerate}
    \item It can be checked directly that $Ag'f'p+Bpqf'+Cg'pq\in\ker\phi_{n-1}^\mu$ if and only if \[A=B\tau^{n-1,m+2}\left(\frac{a+m+3}{a+m-1}\right)^2-\frac{C\sig^{n-1,m-2}(a+m-1)}{a+m+1},\] in which case
        \begin{align*}
            \frac{A(a+m-1)(a+m+1)}{\beta^{n,-m+2}}&pg(q+\beta^{n,-m+2}f+\gamma^{n,-m+2}g)+\\
            \frac{B}{\sig^{n,m}}pq(\sig^{n,m}f'&+\tau^{n,m}g')=Ag'f'p+Bpqf'+Cg'pq.
        \end{align*}

        \item  It can be checked directly that $Apq+Bg'p+Cf'p\in\ker\phi_{n-1}^\mu$ if and only if 
    \begin{align*}
        &B=-\frac{A\beta^{n,-m}}{a+m+3}, \  C=-\frac{A\gamma^{n,-m}}{a+m-1},
    \end{align*}
    in which case 
    \[Ap(q+\beta^{n,-m}f+\gamma^{n,-m}g)=Apq+Bf'p+Cg'p.\]

    \item  It can be checked directly that $Agf+Bqf'+Cqg'\in\ker\phi_{n-1}^\mu$ if and only if \[A=(a-m+1)(a-m-1)\left(B\tau^{n-1,-m+2}\left(\frac{a-m+3}{a-m-1}\right)^2-C\sig^{n-1,-m-2}\right),\] in which case
    \begin{align*}
    \frac{B}{\sig^{n,-m}}q(\sig^{n,-m}f'+\tau^{n,-m}g')&+ \frac{A}{\beta^{n,m+2}}g(q+\beta^{n,m+2}f+\gamma^{n,m+2}g)\\
    &=Agf+Bqf'+Cqg'.
    \end{align*}
\end{enumerate}

    We now provide the same results for $a\geq 1$, noting the same caveat on permissible values of $m$.
    
     \begin{table}[H]
        \centering
        \caption{Combinations for $a\geq1$}\label{tab:ageq1}
            \begin{tabular}{c|c|c}
                    
                path length & target & combination   \\
                \hline
                $2$ & $\mu+m\eps+(n+1)\del, \ m\in M_{n-1}$ & $gf$  \\
                \hline
                $2$ & $\mu'+m\eps+n\del, \ m\in M_{n-2}$ & $g'f', f'p, g'p$  \\
                \hline
                $2$ & $\mu'+m\eps+(n-2)\del, \ m\in M_{n-2}$ & $pq, f'p, g'p$  \\
                \hline
                $2$ & $\mu+m\eps+(n-1)\del, \ m\in M_{n-3}$ & $gf, qf', qg'$  \\
                \hline
                $3$ & $\mu'+m\eps+(n+1)\del, \ m\in M_{n-1}$ & $g'f'p$  \\
                \hline
                $3$ & $\mu+m\eps+n\del, \ m\in M_{n-2}$ & $gfq$  \\
                \hline
                $3$ & $\mu'+m\eps+(n-1)\del, \ m\in M_{n-3}$ & $g'f'p,pqf',g'pq$  \\
                \hline
                 $4$ & $\mu+m\eps+n\del, \ m\in M_{n-2}$ & $g'pqf'$  \\
            \end{tabular}
        \end{table}
    

        Once again, let $A,B,C\in\C$.

    \begin{enumerate}
        \item We have, $Ag'f'+Bf'p+Cg'p\in\ker\phi_{n-1}^\mu$ if and only if \[A=\frac{1}{a-m+1}\left(B\hat{\tau}^{n-1,-m-2}\left(\frac{a-m-1}{a-m+3}\right)^2-C\hat{\sig}^{n-1,-m+2}\right),\]
        in which case 
    
        \begin{align*}
            \frac{A}{\hat{\beta}^{n,m-2}}g'(p+&\hat{\beta}^{n,m-2}f'+\hat{\gamma}^{n,m-2}g') +\\
            &\frac{(a-m-1)B}{\hat{\sig}^{n,-m}}p(\hat{\sig}^{n,-m}f+\hat{\tau}^{n,-m}g) = Ag'f'+Bf'p+Cg'p.
        \end{align*}
    
        \item We have, $Apq+Bf'p+Cg'p\in\ker\phi_{n-1}^\mu$ if and only if 
        \[B=\frac{Av^{n,-m}}{a-m-1}, \ \ C=\frac{Aw^{n,-m}}{a-m+3},\]
        in which case \[Ap(q+v^{n,-m}f+w^{n,-m}g)=Apq+Bf'p+Cg'p.\]
    
        \item  We have, $Agf+Bqf'+Cqg'\in\ker\phi_{n-1}^\mu$ if and only if \[A=(a+m+1)(a+m+3)\left(\left(\frac{a+m-1}{a+m+3}\right)^2B\hat{\gamma}^{n-1,-m+2}-C\hat{\beta}^{n-1,-m-2}\right),\] in which case 
    
        \begin{align*}
            \frac{A}{v^{n,m+2}}g(q&+v^{n,m+2}f+w^{n,m+2}g) + \\
            &\frac{B}{\hat{\beta}^{n,-m}}q(p+\hat{\beta}^{n,-m}f'+\hat{\gamma}^{n,-m}g') =Agf+Bqf'+Cqg'.
        \end{align*}
    
        \item 
    
        Finally, $Ag'f'p+Bpqf'+Cg'pq\in\ker\phi_{n-1}^\mu$ if and only if \[A=\left(\frac{a-m-1}{a-m+3}\right)^2B\hat{\gamma}^{n-1,m+2}-\left(\frac{a-m+3}{a-m+1}\right)C\hat{\beta}^{n-1,m-2},\]
        in which case 
    
        \begin{align*}
        &\frac{A(a-m+1)(a-m+3)}{v^{n,-m+2}}pg(q+v^{n,-m+2}f+w^{n,-m+2}g) + \\
            &\frac{B}{\hat{\beta}^{n,m}}pq(p+\hat{\beta}^{n,m}f'+\hat{\gamma}^{n,m}g') = Ag'f'p+Bpqf'+Cg'pq.
        \end{align*}
       
    \end{enumerate}
    
\end{proof}

\begin{cor}
    We have $\dim\ext^n_{\O_{odd}}(L(\mu),L(\lam))=N^n_n(\mu,\lam)$, as described in Theorem \ref{thrm:extn}.
\end{cor}
\begin{proof}
    Since the equivalence of categories $\O_{odd}\to\A$-Mod (\cite[\S 2.1--2.2]{Mazorchuk_Ovsienko_Stroppel}) identifies $L(\lam)$ with $\C e_\lam$, it follows that 
    \begin{align*}
        \dim\ext_{\O_{odd}}^n(L(\mu),L(\lam)) = \dim\ext^n(\C e_\mu,\C e_\lam) &=
        \sum_{j\in\N}\dim[\mathrm{ext}^n(\C e_\mu,\C e_\lam\left<j\right>)]\\
        &=\dim[\mathrm{ext}^n(\C e_\mu,\C e_\lam\left<n\right>)]\\
        &=N^n_n(\mu,\lam).
        \end{align*}
    The third equality is Koszulity (Theorem \ref{thm:koszul}): as $\P_n^\mu$ is generated in degree $n$, the terms with $j\neq n$ vanish.
\end{proof}

\section{Declaration of generative AI}
During the preparation of this work, the author used Claude to assist with Latex debugging, exposition refinement and proofreading (including to check the manuscript for internal inconsistencies and errors). The author reviewed and edited the output as needed and takes full responsibility for the content of the published article.

\bibliographystyle{alpha}
\bibliography{refs.bib}

\appendix

\section{Details on the description of \texorpdfstring{$\A$}{A}}\label{ap:0}

To see how we arrived at the definition of $f'_\mu$ and $g'_\mu$, for each  $\mu=a\eps+b\del\in\Lambda_{odd}$, let us start by assuming that $\alpha_{-a}$ and $\beta_{-a-4}$ are arbitrary in $\C$. 

The restrictive part (with the goal of proving that $\A$ is Koszul) is finding a choice of $\alpha_{-a},\beta_{-a-4}$ such that $\ker\Phi$ is generated by quadratic relations. Let $r$ be the morphism induced by $(C)$ in Proposition \ref{prop:targets_verm}. Let $r_1'$ be the morphism induced by $(C_1)$ and $r_2'$ be the morphism induced by $(C_2)$ in Proposition \ref{prop:targets_non_verma}.  Then \textit{Mathematica} helps us derive the relations in the following proposition. 
\begin{prop}\label{prop:relns_downstairs}
    Fix $\mu=a\eps+b\del$. In what follows, $f_i'g_j'$ corresponds to 
    \[P(\mu')\xrightarrow{ \ g_j'\ } P(\mu'+2\eps+\del)\xrightarrow{ \ f_i'\ } P(\mu'+2\del),\] and $g_i'f'_j$ corresponds to \[P(\mu')\xrightarrow{ \ f_j'\ } P(\mu'-2\eps+\del)\xrightarrow{ \ g_i'\ } P(\mu'+2\del),\] for $i,j=1,2$.
    
    \begin{enumerate}[label=(\alph*)]
        \item $f_2'g_2'=-(a-1)((a-1)(a+1)r_2'+2r_1')$.
        \item $g_2'f_2'=(a+3)((a+3)(a+1)r_2'+2r_1')$.
        \item $g_2'f_1'=-g_1'f_2'=(a+3)^2r_1'$.
        \item $f_2'g_1'=-f_1'g_2'=(a-1)^2r_1'$.
        \item $f_1'g_1'=0=g_1'f_1'$.
    \end{enumerate}
    When $a=1$, all the same relations hold except that $(a)$ and $(d)$ are replaced by $f_2'g_2'=2r_2'$ and $f_1'g_2'=-f_2'g_1'=r_1'$ respectively.
    When $a=-1$, we have the following relations:
    \begin{enumerate}[label=(\alph*)]
        \item $g_1'f_1'=0.$
        \item $g_2'f_1'=g_2'f_2'=-4r.$
        \item $g_1'f_2'=4r.$
    \end{enumerate}
    Note the equality to $r$ (rather than $r_1',r_2'$) because the codomain is $P(-\eps+(b+2)\del)$, which is a Verma module.
\end{prop}
\begin{proof}
    All computed directly using Mathematica.
\end{proof}

When $a\geq3$, if we took $f'=f_2'$ and $g'=g_2'$, a non-quadratic relation of the form $f_2'g_2'+\gamma g_2'f_2'=\sigma pqg_2'f_2'$ can be derived and cannot be reduced to a quadratic relation. However, wonderfully (for the claim of Koszulity), there exists a choice of $\alpha$ and $\beta$ in Definition \ref{def:generators_of_A} such that $f'g'+cg'f'=0$ for a scalar $c$.

\begin{lem}\label{lem:diff_eq}
    For $a\in\Z$ with $a\neq1,-1,-3$, the difference equation \[\beta_{a-2}-\beta_{a}-\alpha_{a}+\alpha_{a+2}=\frac{8}{(a-1)(a+3)}\] is solved by \[\alpha_a = -\frac{2}{a+1}, \ \ \beta_a=\frac{2}{a+1}. \]
\end{lem}
\begin{proof}
    Observe that
    \begin{align*}
        \frac{8}{(a-1)(a+3)}
        =\frac{2}{(a-1)}-\frac{2}{a+1}-\left(-\frac{2}{a+1}\right)-\frac{2}{(a+3)},
    \end{align*}
    then the claim becomes obvious. 
\end{proof}

\begin{prop}\label{prop:f'g'}
    Define $f_\mu,g_\mu,p_\mu,q_{\mu'},f'_{\mu'},g'_{\mu'}$ as in Definition \ref{def:generators_of_A} for all $\mu\in\Lambda_{odd}$. Then all the relations outlined in Theorem \ref{thm:relations_in_A} hold.
\end{prop}
\begin{proof}
    Fix $\mu=a\eps+b\del\in\Lambda_{odd}$ such that $a\geq3$, and consider the expansion of $\frac{1}{(a-1)^2(a+1)}f'g' + \frac{1}{(a+3)^2(a+1)}g'f'$ with careful consideration of the source of each generator; $f'g'$ corresponds to
    \[P(\mu')\xrightarrow{ \ g' \ }P(\mu'+2\eps+\del)\xrightarrow{ \ f' \ }P(\mu'+2\del),\]
    and $g'f'$ corresponds to
     \[P(\mu')\xrightarrow{ \ f' \ }P(\mu'-2\eps+\del)\xrightarrow{ \ g' \ }P(\mu'+2\del).\]

     By expanding as above and using Proposition \ref{prop:relns_downstairs}, we see that this sum is $0$ if and only if \[\beta_{-a-4}-\beta_{-a-2}-\alpha_{-a-2}+\alpha_{-a}=\frac{8}{(a-1)(a+3)}.\] If we map $a\mapsto -a-2$ then we get precisely the difference equation in Lemma \ref{lem:diff_eq}. 

    The cases of $a=\pm1$ can be verified similarly, and the relations in Theorem \ref{thm:relations_in_A} are verified directly using Mathematica. 

    We also spell out the expansion and simplification of the $f',g'$ relation for the local quiver centred at $\mu=a\eps+b\del$ when $a=1$. Here, the $g'$ in $f'g'$ is a morphism $P(\mu')\to P(\mu'+2\eps+\del)=\Del(-\eps+(b+1)\del)$ by Remark \ref{rmk:gen_of_projs} and so it is generated by $(D_2)$ in Proposition \ref{prop:targets_verm}. Then, a similar expansion to the $a\geq3$ case yields
    \begin{align*}
        f'g'+cg'f' &=(f_2'-f_1')g_2'+c(g_2'+g_1')\left(f_2'-\frac{1}{2}f_1'\right)\\
        &=2(1+16c)r_2'-(1+16c)r_1'
    \end{align*}
    which is $0$ if and only if $c=-\frac{1}{16}$. 
    
\end{proof}

\section{Proof of the $a\leq-1$ case in Lemma \ref{lem:coefficients}}\label{ap:1}

We provide the full details on computing $(G_{n-1}^\mu)^n$ for all $n\geq 1$, $\mu=a\eps+b\del\in\Lambda_{odd}$ and $a\leq -1$. The case of $a\geq 1$ follows from a similar argument and is omitted for brevity.

The structure of this proof for the $a\leq -1$ case is as follows. 

\begin{itemize}
    \item The base case $n=1$ is trivial.
    \item Proceed by induction, first assuming $2n<1-a$, using the relations of $\A$ to derive the relevant relations and show there are no other elements in $(G_{n-1}^\mu)^n$.
    \item Finally show that a slight augmentation of the induction holds when $2n\geq 1-a$. 
\end{itemize}

\begin{proof}
For $\mu=a\eps+b\del\in\Lambda_{odd}$ the base cases in Lemma \ref{lem:coefficients} can be easily verified by drawing appropriate diagrams as in Figure \ref{fig:extensions_comp}. Note that when $a=-1$, the second row in the base case $n=1$ degenerates as $M_0^{>2(1-1)}=\varnothing$.

Now let $a\leq-1$. Proceed by induction for $n\geq2$, assuming first that $2n<1-a$. The induction will also hold for $2n\geq 1-a$ after a slight augmentation. 

By Lemma \ref{lem:which_combinations} and the inductive hypothesis for $n-1$ (that $(G_{n-2}^\mu)^{n-1}$ is as described in Lemma \ref{lem:coefficients}), one can verify that \[(\ker\phi_{n-1}^\mu)_n=\oplus_{\lam\in S_n^\mu} e_\lam(\ker\phi_{n-1}^\mu)_n\] where \[S_n^\mu=\{\mu+m_1\eps+n\del \ | \ m_1\in M_n\}\cup \{\mu'+m_2\eps+(n-1)\del \ | \ m_2\in M_{n-1}\}.\] This vector space is spanned by the set of vectors $\sig^{n,m_1}f'+\tau^{n,m_1}g'\in e_{\mu+m_1\eps+n\del}(\ker\phi_{n-1}^\mu)_n$ and $\alpha^{n,m_2}q+\beta^{n,m_2}f+\gamma^{n,m_2}g\in e_{\mu'+m_2\eps+(n-1)\del}(\ker\phi_{n-1}^\mu)_n$ over all $m_1\in M_n$ and $m_2\in M_{n-1}$. We will show there is exactly one (non-degenerate) choice of $\sig,\tau$ and $\alpha,\beta,\gamma$ (up to scalar) for each $m_1$ and $m_2$, and show we can set $\alpha=1$ without loss of generality.

\begin{figure}[H]
\minipage{0.5\textwidth}
\adjustbox{scale=.45,center}{
    \begin{tikzcd}
    &    & b+n-2  & b+n-1 & b+n                    \\
    & a +4+m_1   &  \phantom{\circ}   \arrow[dr, dotted]&  &      &                      \\
    &a+2+m_1  &  & \circ   \arrow[dr, "f'"]   &                 \\
    &a+m_1 & \phantom{\circ}   \arrow[ur, dotted]\arrow[dr, dotted] &    &  \bullet                   \\
    &a-2+m_1 & &  \circ  \arrow[ur, "g'", swap]  &                 \\
    &a-4+m_1 & \phantom{\circ}   \arrow[ur, dotted]&  &              
\end{tikzcd}
    }
\caption{$n>3$, $f',g'$ relation}
\label{fig:n>3, f'g'}
\endminipage\hfill
\minipage{0.5\textwidth}%
\adjustbox{scale=.45,center, }{
    \begin{tikzcd}
    &    & b+n-3  & b+n-2 & b+n-1          & b+n               \\
    &-a+2 + m_2  & \phantom{\circ}   \arrow[dr, dotted] &          &    &          \\
    &-a+m_2  & &          \circ \arrow[dr, "g"]      &          &           \\
    &-a-2+m_2   & \phantom{\circ}   \arrow[ur, dotted]  \arrow[dr, dotted] &   &  \bullet   &              \\
    &-a-4+m_2  & &  \circ \arrow[ur, "f", swap] &        &               \\
    &-a-6+m_2  & \phantom{\circ}   \arrow[ur, dotted] &   &         &           \\
    &  \vdots   & &  &         &                         \\
    & \vdots    & &  &         &                      \\
    & a +4-m_2  &  &  &      &                      \\
    &a+2-m_2  &  & \phantom{\circ}   \arrow[dr, dotted] \arrow[uuuuu,bend right=15, dotted]  &          &  \phantom{\circ}           \\
    &a-m_2 &  &    &  \circ \arrow[uuuuuuu,bend right=15, "q", swap] &                  \\
    &a-2-m_2 & &  \phantom{\circ}   \arrow[ur, dotted] \arrow[uuuuuuuuu,bend right=15, dotted] &    \phantom{\circ}    &            \\
    &a-4-m_2 & &  &        &             
\end{tikzcd}
}
\hspace{1pt}
\caption{$n>3$, $q,f,g$ relation}
\label{fig:n>3, q,f,g}
\endminipage
\end{figure}

Let $f',g'\in e_\lam\P_{n-1}^\mu$ where $\lam=\mu+m_1\eps+n\del$ for $m_1\in M_n$. By induction, the only choices (up to scalar) of $\sig,\tau$ when $m_1=\pm2n$ and $n\geq2$ are
$\sigma^{n,2n}=\tau^{n,-2n}=0$ and $\sigma^{n,-2n}=\tau^{n,2n}=1$. We now find $\sig,\tau$ for $m_1\in M_{n}\setminus\{\pm2n\}$, corresponding to the $f',g'$ relation marked in Figure \ref{fig:n>3, f'g'}. 

For $\sigma^{n,m_1} f'+\tau^{n,m_1} g'\in\ker\phi_{n-1}^\mu$, we must have (by induction)
\[\sigma^{n,m_1}\tau^{n-1,m_1+2}f'g'+\tau^{n,m_1}\sigma^{n-1,m_1-2}g'f'=0,\] which, by considering the relations between $f',g'\in\A$, is satisfied by non-zero $\sig,\tau$ if and only if (up to scalar) \begin{equation}\label{eq:sig_tau_rel}\sigma^{n,m_1}=\frac{1}{\tau^{n-1,m_1+2}}, \ \ \tau^{n,m_1} = \frac{1}{\sigma^{n-1,m_1-2}}\left(\frac{a+m_1+3}{a+m_1-1}\right)^2.\end{equation}

Observe that $a+m_1-1\neq 0$ because we have assumed $2n<1-a$ so $M_n=M^{<1-a}_n$ and $m_1<1-a$. Further, $\tau^{n-1,m_1+2}\neq0$ and $\sig^{n-1,m_1-2}\neq0$ by induction.

Now let $q,f,g\in e_\lam\P_{n-1}^\mu$ where $\lam=\mu'+m_2\eps+(n-1)\del$ with $m_2\in M_{n-1}$. We find unique $\alpha,\beta,\gamma$ for $q,f,g$ depicted in Figure \ref{fig:n>3, q,f,g}.

By induction, $\alpha^{n,m_2}q+\beta^{n,m_2}f+\gamma^{n,m_2}g\in\ker\phi_{n-1}^\mu$ if and only if 
\begin{align*}
    &\alpha^{n,m_2}q\left( \sigma^{n-1,-m_2}f'+\tau^{n-1,-m_2}g' \right) + \\[0.5em]
    &\beta^{n,m_2}f\left(q+\beta^{n-1,m_2-2}f+\gamma^{n-1, m_2-2}g\right) + \\[0.5em]
    &\gamma^{n,m_2}g\left(q+\beta^{n-1,m_2+2}f+\gamma^{n-1, m_2+2}g\right) = 0.
\end{align*}

Re-grouping in terms of corresponding relations in $\A$, we see that this linear combination is $0$ if and only if
\begin{align*}
    &(1) \ \ \ \ \beta^{n,m_2}=(a-m_2+1)\alpha^{n,m_2}\sigma^{n-1,-m_2},\\[0.5em]
    &(2) \ \ \ \ \gamma^{n,m_2}=(a-m_2+1)\alpha^{n,m_2}\tau^{n-1,-m_2},\\[0.5em]
    &(3) \ \ \ \ \gamma^{n,m_2}\beta^{n-1,m_2+2}=\left(\frac{a-m_2+3}{a-m_2-1}\right)\beta^{n,m_2}\gamma^{n-1,m_2-2}.
\end{align*}

Hence, if $\alpha^{n,m_2}=0$ then $\beta^{n,m_2}=\gamma^{n,m_2}=0$, so for a non-degenerate choice of coefficients we need $\alpha^{n,m_2}\neq0$. In particular, we can assume (without loss of generality, up to scalar) that $\alpha^{n,m_2}=1$ for all such $m_2\in M_{n-1}$. 

Note also that with $2n<1-a$ we have $M_{n-1}=M_{n-1}^{>2(n+a)}$, and
$a-m_2-1\neq0$ since $a-1<-2n\leq-2(n-1)\leq m_2$; likewise
$-m_2\leq2(n-1)<1-a$, so $\sig^{n-1,-m_2}$ and $\tau^{n-1,-m_2}$ are defined.

By \eqref{eq:sig_tau_rel}, for non-extremal $m_2$ we know that $\sigma^{n-1,-m_2}=\frac{1}{\tau^{n-2,-m_2+2}}$, $\tau^{n-1,-m_2} = \frac{1}{\sigma^{n-2,-m_2-2}}\left(\frac{a-m_2+3}{a-m_2-1}\right)^2$. Using these identities and our inductive assumption, we can eliminate $\sigma$ and $\tau$ from the first two equations to get the following two equations,
\begin{align*}
    &\beta^{n,m_2}\gamma^{n-1,m_2-2}=(a-m_2+1)(a-m_2+3),\\[0.5em]
    &\gamma^{n,m_2}\beta^{n-1,m_2+2}=\frac{(a-m_2+1)(a- m_2+3)^2}{(a-m_2-1)},
\end{align*}

which are consistent with equation (3) (at the extremal values $m_2=\pm2(n-1)$, both sides of (3) vanish). Hence, the only choice (up to scalar) is \[\alpha^{n,m_2}=1, \  \beta^{n,m_2}=(a-m_2+1)\sig^{n-1,-m_2}, \ \gamma^{n,m_2} = (a-m_2+1)\tau^{n-1,-m_2}\] for all $n\geq2$ and $m_2\in M_{n-1}$. This concludes the first half of the proof, that case (1) of this lemma holds for $n\geq2$ and $2n < 1-a$.

The inductive argument described above can be amended slightly so that it also works when $2n\geq1-a$. The three changes to the argument are that the constraints $m_1<1-a$ and $m_2>2(n+a)$ are no longer vacuous; the values of $\sig^{n,m_1^{\max}}$ and
$\tau^{n,m_1^{\max}}$ must be read off the local quivers rather than the recurrence relations, and although $\beta^{n,m_2},\gamma^{n,m_2}$ are still defined by the same relations, we need to verify why the corresponding combination is not in the kernel for $m_2\leq 2(n+a)$. Other than this, the construction in the $2n<1-a$ case passes straight through.

\begin{figure}[H]
\minipage{0.5\textwidth}
\adjustbox{scale=.45,center}{
       \begin{tikzcd}
    &    & b+n-3  & b+n-2 & b+n-1          & b+n               \\
     &\vdots  & \vdots      &  \vdots     &  \vdots &  \vdots        \\
    &9 & &   \circ \arrow[dr]   &       &            \\
    &7   &  \phantom{\circ} \arrow[dr, dotted]\arrow[ur,dotted]  &  &   \bullet  &              \\
    &5   &     & \circ\arrow[ur] \arrow[dr] &    &              \\
    &3  & \phantom{\circ} \arrow[dr, dotted] \arrow[ur, dotted]  &    &  \bullet      &              \\ 
    &1 & \phantom{\circ}   &  \circ \arrow[ur] &         &           \\
    & -1   &  &   & \circ\arrow[dr]     &                      \\
    & -3 &  &   \phantom{\circ}  \arrow[dr, dotted] \arrow[ur, dotted] \arrow[uu, bend right = 15, dotted]   &         &  \bullet          \\
    &-5 &  &    & \circ\arrow[ur] \arrow[dr] \arrow[uuuu, bend right=15]   &                  \\
    &-7  & &            \phantom{\circ}  \arrow[dr, dotted] \arrow[ur, dotted] \arrow[uuuuuu, bend right = 15, dotted]   &          &   \bullet        \\
    &-9  & &               &        \circ \arrow[ur]\arrow[uuuuuuuu, bend right=15]  &           \\
     &\vdots  & \vdots      &  \vdots \arrow[ur, dotted] \arrow[uuuuuuuuuu, bend right = 15, dotted]    &  \vdots &  \vdots        \\
    & & &  &        &            
\end{tikzcd}
    }
\caption{$2n=1-a$}
\label{fig:2n=1-a}
\endminipage\hfill
\minipage{0.5\textwidth}%
\adjustbox{scale=.45,center, }{
    \begin{tikzcd}
    &   & b+n-3 & b+n-2          & b+n-1     & b + n        \\
    &\vdots  & \vdots      &  \vdots    &  \vdots &  \vdots        \\
    &\vdots  & \vdots      &  \vdots    &  \vdots &  \vdots        \\
    &7   &  &   \circ   \arrow[dr] &          &    \\
    &5        & \phantom{\circ}  \arrow[ur, dotted] \arrow[dr, dotted] &    &  \bullet  &          \\
    &3  &    &  \circ   \arrow[ur] \arrow[dr, color=red]  &       &       \\ 
    &1    & \phantom{\circ}  \arrow[ur, dotted] &         &   {\color{red}\mathord{?}}  &      \\
    & -1    &   & \phantom{\circ} \arrow[dr, dotted]     &    & \bullet                 \\
    & -3  &   \phantom{\circ}  &         &  \circ  \arrow[ur] \arrow[dr]  &    \\
    &-5   &    & \phantom{\circ}  \arrow[ur, dotted] \arrow[dr, dotted] \arrow[uuuu, bend right=15, dotted]   &         &  \bullet       \\
    &-7  &            \phantom{\circ}   &          &   \circ \arrow[ur]\arrow[uuuuuu, bend right = 15]   &     \\
    &-9   &               &    \phantom{\circ} \arrow[ur, dotted] \arrow[uuuuuuuu, bend right = 15, dotted]    &      &     \\
     &\vdots  & \vdots      &  \vdots     &  \vdots &  \vdots        \\
    &  &  &        &         &
\end{tikzcd}
}
\hspace{1pt}
\caption{$2n=3-a$}
\label{fig:2n=3-a}
\endminipage
\end{figure}

It is easiest to see what is going on by considering the smallest (in $n$)
instances where these changes start.

For $f',g'$ and $\lam=\mu+m_1\eps+n\del$: when $2n=1-a$, there are no edges in
$Q$ from $-\eps+(b+n-1)\del\to\eps+(b+n)\del$, so there is no element of
$e_{\mu+2n\eps+n\del}(\ker\phi_{n-1}^\mu)_n$ (see Figure \ref{fig:2n=1-a}).
Moreover, we have $a+m_1^{\max}=-3$, and the relevant
$f',g'$ relation is the one in the local quiver at $a=1$, namely
$f'g'=\frac{1}{16}g'f'$. The combination therefore lies in the kernel if and
only if
\[\sig^{n,m_1^{\max}}\tau^{n-1,m_1^{\max}+2}\tfrac{1}{16}
   +\tau^{n,m_1^{\max}}\sig^{n-1,m_1^{\max}-2}=0,\]
so up to scalar $\sig^{n,m_1^{\max}}=1/\tau^{n-1,m_1^{\max}+2}$ and
$\tau^{n,m_1^{\max}}=-1/(16\sig^{n-1,m_1^{\max}-2})$. This is the recurrence
relation with the factor $\left(\frac{a+m_1+3}{a+m_1-1}\right)^2$, which
vanishes when $a+m_1=-3$, replaced by $-\frac{1}{16}$. Iterating to $2n=3-a$, we have
$a+m_1^{\max}=-1$, and now the $g'f'=0$ relation in the local quiver at $a=-1$
means that $g'\in e_{\mu+m_1^{\max}\eps+n\del}(\ker\phi_{n-1}^\mu)_n$, so that
$\sig^{n,m_1^{\max}}=0$, $\tau^{n,m_1^{\max}}=1$. Evidently this pattern continues for higher $n$. This is why we filter
$m_1<1-a$ (independent of $n$), with the choices of
$\sig^{n,m_1^{\max}},\tau^{n,m_1^{\max}}$ depending on whether
$a+m_1^{\max}=-1$ or $-3$.

Similarly, for $q,f,g$ and $\lam=\mu'+m_2\eps+(n-1)\del$: when $2n=1-a$, this
combination does not exist for $m_2=-2(n-1)$, because there is no edge in $Q$
from $\eps+c\del\to-\eps+(c+1)\del$ for any $c\in\Z$ (see Figure
\ref{fig:2n=1-a}). Looking then at the $2n=3-a$ case in Figure
\ref{fig:2n=3-a}, the combination again does not exist for $m_2=-2(n-1)$, and
it is not in the kernel when $m_2=-2(n-1)+4$, since there is nothing to cancel
out the $fg$ term in the expansion. Inducting on $n$, the combination at
$m_2$ involves those at $m_2\pm2$ one level down in the induction. When $m_2<2(n+a)$ we have
$m_2\pm2\leq2\big((n-1)+a\big)$, so neither is in the kernel and the
combination reduces to $\alpha^{n,m_2}q$, which is not in the kernel as above;
and when $m_2=2(n+a)$ there is nothing to cancel the $fg$ term. Hence the
linear combination is not in the kernel for all $m_2\leq2(n+a)$. For
$m_2>2(n+a)$, the coefficients $\beta,\gamma$ satisfy the same relations as in
the case $2n<1-a$. This concludes the proof of case (1) of Lemma
\ref{lem:coefficients}.

The case $a\geq1$ follows by the same argument, applied to the three families
listed in Lemma \ref{lem:coefficients}. The third of them,
$q+v^{n,m_3}f+w^{n,m_3}g$, has no analogue when $a\leq-1$: it exists because of
the relation $qp=0$. It does not exist for $n=1$ (corresponding to $M_{-1}=\varnothing$), and for
$n=2$ consists of the single element with $m_3=0$; this is why its base case is
set at $n=2$. The remaining details follow similarly to the $a\leq -1$ case, following more closely the $q,f,g$ argument beyond $2n\geq 1+a$, since the $g'f'=0$ relation does not cause the same pattern (as seen between $f',g'$ when $a\leq -1$) for the $f',g',p$ relation precisely because $g'p\neq0$. 



\end{proof}

\end{document}